\documentclass[reqno]{amsart}

\usepackage[utf8]{inputenc}
\usepackage[english]{babel}

\usepackage{amsmath}
\usepackage{amsfonts}
\usepackage{amssymb}
\usepackage{amsthm}
\usepackage{mathrsfs}
\usepackage{listings}
\usepackage[normalem]{ulem}
\usepackage{hyperref}         
\usepackage{algpseudocode} 
\usepackage[tight,footnotesize]{subfigure}
\usepackage[font=small,labelfont=bf]{caption}
\usepackage{mathrsfs,textcomp}
\usepackage{bbm}
\usepackage{pdflscape}
\usepackage{graphicx}
\usepackage{tikz-cd}
\usepackage{arydshln} 
\usepackage{multirow}
\usepackage{breakurl}  
\usepackage{soul}
\usepackage{color}

\makeatletter
\newtheorem*{rep@theorem}{\rep@title}
\newcommand{\newreptheorem}[2]{%
\newenvironment{rep#1}[1]{%
 \def\rep@title{#2 \ref{##1}}%
 \begin{rep@theorem}}%
 {\end{rep@theorem}}}
\makeatother

\newtheorem{thm}{Theorem}[section]
\newreptheorem{thm}{Theorem}
\newtheorem{prop}[thm]{Proposition}
\newreptheorem{prop}{Proposition}
\newtheorem*{thm:main}{Theorem \ref{main theorem}}
\newtheorem*{thm:main2_gen3}{Theorem \ref{cycliclist}}
\newtheorem*{thm:main1_gen2}{Theorem \ref{thm: non-normallist}}
\newtheorem*{thm:main2_gen2}{Theorem \ref{thm: cycliclist}}
\newtheorem{cor}[thm]{Corollary}

\newtheorem{remk}[thm]{Remark}

\newtheorem{example}[thm]{Example}
\newtheorem{Algorithm}[thm]{Algorithm}
\numberwithin{equation}{section}
\theoremstyle{definition}

\newtheorem{lem}[thm]{Lemma}

\newcommand{\RR}{\mathbb{R}}      
\newcommand{\ZZ}{\mathbb{Z}}   
\newcommand{\CC}{\mathbb{C}}      
\newcommand{\QQ}{\mathbb{Q}}      

\newcommand{\ord}{\text{\normalfont{ord}}}

\newcommand{\Homo}{\text{\normalfont{Hom}}}
\newcommand{\OO}{\mathcal{O}}

\newcommand{\End}{\text{\normalfont{End}}}
\newcommand{\Pic}{\text{\normalfont{Pic}}}
\newcommand{\Gal}{\text{\normalfont{Gal}}}

\newcommand{\N}{\text{\normalfont{N}}}

\newcommand{\id}{\text{\normalfont{id}}}
\renewcommand{\mod}{\text{\normalfont{mod}}}

\newcommand{\bb}{\mathfrak{b}}
\newcommand{\aaa}{\mathfrak{a}}

\newcommand{\p}{\mathfrak{p}}
\newcommand{\lp}{\mathfrak{l}}

\newcommand{\q}{\mathfrak{q}}

\newcommand{\Cl}{\text{\normalfont{Cl}}}
\newcommand{\im}{\text{\normalfont{Im}}}
\newcommand{\Tr}{\text{\normalfont{Tr}}}
\newcommand{\I}{\text{\normalfont{I}}}
\newcommand{\D}{\text{\normalfont{D}}}

\newcommand{\Tgt}{\text{\normalfont{Tgt}}}

\newcommand{\thickhline}{%
    \noalign {\ifnum 0=`}\fi \hrule height 1pt
    \futurelet \reserved@a \@xhline
}

\title{The CM class number one problem for curves of genus $2$}

\author{P{\i}nar K{\i}l{\i}\c{c}er,
Marco Streng
}

\date{\today}

\newcommand{\Addresses}{{
  \bigskip
  \footnotesize

 P{\i}nar K{\i}l{\i}\c{c}er, \textsc{Carl von Ossietzky Universität Oldenburg 
Institut für Mathematik 
26111 Oldenburg, Germany}\par\nopagebreak
  \textit{E-mail address}, P{\i}nar K{\i}l{\i}\c{c}er: \texttt{pinar.kilicer@uni-oldenburg.de}

  \medskip

 Marco Streng, \textsc{Mathematisch Instituut
Universiteit Leiden
P.O. box 9512
2300 RA Leiden
The Netherlands}\par\nopagebreak \textit{E-mail address}, Marco Streng: \texttt{streng@math.leidenuniv.nl}
  
}}

\begin{document}

\begin{abstract}The CM class number one problem for elliptic curves asked to find all elliptic curves defined over the rationals with non-trivial endomorphism ring. For genus-2 curves it is the problem of determining all CM curves of genus $2$ defined over the \textit{reflex field}. We solve the problem by showing that the list given in Bouyer--Streng \cite[Tables 1a, 1b, 2b, and 2c]{bouyer} is complete.
\end{abstract} 
\maketitle

\section{Introduction}

By a \textit{curve}, we always mean a smooth projective geometrically irreducible algebraic curve
over a field~$k$.
A \emph{CM field} is a totally imaginary quadratic extension of a totally real number field.
We say that a curve $C$ of genus $g$ has \textit{complex multiplication} (CM) if there
exists an embedding $\theta$ of a CM field $K$ of degree $2g$ into
the endomorphism algebra $\mathrm{End}(J(C)_{\overline{k}})\otimes \QQ$
of the Jacobian over the algebraic closure.
We then also say that $C$ has \emph{CM by the order} $\OO= \theta^{-1}(\mathrm{End}(J(C)_{\overline{k}}))$.
For example, an elliptic curve~$E/\overline{\QQ}$ has CM if
$\End(E)$ is an order in an imaginary quadratic field~$K$. 
An elliptic curve $E/\overline{\QQ}$ with CM by a maximal order $\OO_{K}$ can be defined over $\QQ$ if and
only if the class group $\Cl_{K}:=I_{K}/P_{K}$ is trivial. 
The CM class number one problem for elliptic curves asks to determine all imaginary quadratic fields of class number one. 
This problem was solved by Heegner~\cite{heegner} (1952), Baker \cite{baker} (1966) and Stark \cite{stark} (1967); the fields are $K\cong\QQ(\sqrt{-d})$ where $d\in \{1,\ 2,\ 3,\ 7,\ 11,\ 19,\ 43,\ 67,\ 163\}$.

We consider the analogous problem for curves of genus $2$.  
Murabayashi and Umegaki~\cite{umeg} listed all quartic CM fields $K$
for which there exist
curves of genus $2$ over $\QQ$ with CM by~$\OO_K$. This list  contains only cyclic
quartic CM fields, and no dihedral quartic CM fields because curves
cannot be defined over~$\QQ$ in the dihedral case. 
The smallest examples in the generic dihedral case are the ones defined over the reflex
field, or equivalently those whose \textit{CM class group} $I_{K^r}/I_0(\Phi^r)$
is trivial. We call the order of this group the \textit{CM class number}.
We give the complete list of CM class number one non-biquadratic quartic  fields,
thereby solving the \emph{CM class number one problem for curves of genus~$2$}
and showing that the list in Bouyer--Streng \cite{bouyer} is complete. 

\begin{thm}\label{thm: non-normallist} There exist exactly $63$ isomorphism classes of  non-normal quartic CM fields with CM class number one. The fields are listed in Theorem \ref{thm: non-normallist2}.
\end{thm}

\begin{thm}\label{thm: cycliclist}  There exist exactly $20$ isomorphism classes of cyclic quartic CM fields with CM class number one. The fields are listed in Theorem \ref{thm: cycliclist2}.
\end{thm}

\begin{cor} \label{cor: curves}
There are exactly $125$ curves of genus $2$, up to isomorphism over $\overline{\QQ}$, defined over the reflex field with CM by $\OO_K$ for some non-biquadratic quartic  CM field $K$. The fields are the fields in Theorems \ref{thm: non-normallist} and \ref{thm: cycliclist}, and the curves are those of Bouyer--Streng \cite[Tables 1a, 1b, 2b, and 2c]{bouyer}.
\end{cor}
\begin{proof} Theorems \ref{thm: non-normallist} and \ref{thm: cycliclist} give exactly the fields
	and the main result of Bouyer--Streng~\cite{bouyer} is the complete list of curves for those fields.
\end{proof}

\begin{cor}
There are exactly $21$ absolutely simple CM curves of genus $2$ defined over~$\QQ$, up to isomorphism over $\overline{\QQ}$. The fields and $19$ of the curves are given in van Wamelen \cite{wame}. The other two curves are $y^2 = x^6-4x^5+10x^3-6x-1$ and $y^2 = 4x^5+40x^4-40x^3+20x^2+20x+3$ given in Theorem 14 of Bisson--Streng \cite{bisson}.
\end{cor}
\begin{proof}
	For such a curve $C$, let $\OO = \mathrm{End}(J(C)_{\overline{\QQ}})$
	and $K = \OO\otimes \QQ$.
	Then Proposition~5.17 in Shimura~\cite{shimura1} shows that $K/\QQ$ is Galois,
	and hence Theorem~\ref{thm: cycliclist} gives all possibilities for~$K$.
	
	Bouyer-Streng~\cite{bouyer} prove that the 19 curves in \cite{wame} are
	exactly the curves with $\OO=\OO_K$ for these fields~$K$
	and Bisson-Streng~\cite{bisson} prove that the two curves
	in the statement are exactly the curves with $\OO\subsetneq\OO_K$
	for these fields~$K$.
\end{proof}

\begin{cor}
There are only finitely many simple CM curves of genus $2$ defined over the reflex field. The corresponding CM fields are those of Theorems~{\ref{thm: non-normallist}--\ref{thm: cycliclist},} the complete list of orders can be computed using the methods of \cite{bisson} and the curves using the methods of \cite{bouyer}. \qed
\end{cor}
In Section \ref{section: CM}, we give some definitions and facts from CM theory, and state the CM class number one problem for curves of genus $g\leq 2$.
In Sections \ref{non-normal-partone}--\ref{sec:enumerating}, we prove Theorem~\ref{thm: non-normallist}.
The plan is as follows.
We first show that there are only finitely many non-biquadratic quartic  CM fields
with CM class number one by bounding their discriminant (Section~\ref{non-normal-partone}).
The bound will be too large for practical purposes,
but by using ramification theory and $L$-functions,
we improve the bound (Section~\ref{non-normal-parttwo}),
which we then use to enumerate the CM fields (Section~\ref{sec:enumerating}).
Section~\ref{cyclic} proves Theorem~\ref{thm: cycliclist}
using the same strategy as in Sections~\ref{non-normal-partone}--\ref{sec:enumerating}.

\subsection*{Acknowledgements}

 The authors would like to thank Maarten Derickx,
 Florian Hess, and Peter Stevenhagen for useful discussions.

\section{Complex Multiplication}\label{section: CM}
We refer to Shimura \cite{shimura3} and Lang \cite{lang} as references for this section.

Let $K$ be a CM field of degree $2g$ and $N'$ be a number field that contains a subfield isomorphic (over $\QQ$) to a normal closure of $K$. There exists an automorphism $\rho$ of~$K$ such that for every embedding $\tau: K \rightarrow \CC$ we have $\overline{\cdot}\circ \tau = \tau \circ \rho$; we call it \textit{complex conjugation} and denote it by $\rho$ or $\overline{\cdot}$. Let $\phi$ be an embedding of $K$ into any field $N$. 
Then we denote $\phi\circ\overline{\cdot}$ by $\overline{\phi}$. Note that if $N$ 
is a CM field or $\CC$ then we have $\overline{\cdot}\circ\phi=\overline{\phi}$.

Let $N'$ be a field that contains a  
subfield that is isomorphic over $\QQ$ to a normal closure over~$\QQ$ of $K$.  A \textit{CM type} \label{CMtype} of $K$ with values in $N'$ is a set $\Phi$ of 
embeddings $\phi: K\rightarrow N'$ such that exactly one embedding  of each of the  
$g$ complex conjugate pairs $\phi, \ \overline{\phi}: K \rightarrow N'$  is in~$\Phi$.  
We say that $(K,\Phi)$ is a \textit{CM pair} or \textit{CM type}. 

Let $K_1$ be a  proper CM subfield of $K$. Let $\Phi_{1}$ be a CM type of $K_1$ with values 
in $N'$. Then the CM type of $K$  \textit{induced} by 
$\Phi_{1}$ is $\{\phi\in\Homo(K, N'):\ \phi|_{K_1}\in\Phi_{1}\}.$
We say that a CM type $\Phi$ of a CM field is \textit{primitive} if it is not induced 
from a CM type of a proper CM subfield. 

Let $N$ be a normal closure of $K$. From now on, we identify $N'$ with $N$ by making a choice of isomorphism and replacing $N'$ by a subfield. The \textit{reflex field} of $(K,\ \Phi)$ is 
         \[K^{r}=\QQ(\{\sum_{\phi\in\Phi}\phi(x)\ |\ x\in K\}) \subset N \]
and satisfies $\Gal(N/K^r)=\{\sigma\in\Gal(N/\QQ):\sigma\Phi=\Phi\}$. 
For example, if a CM~field~$K$ is Galois over $\QQ$, then the reflex field $K^r$ of $K$ is a subfield of $K$. Let~$\Phi_N$ be the CM type of $N$ induced by $\Phi$. Then the reflex field $K^r$ is a CM field with CM type $\Phi^{r}=\{\sigma^{-1}|_{K^r}:\ \sigma\in\Phi_N\}.$ The pair $(K^r,\Phi^r)$ is called the \textit{reflex} of $(K,\Phi)$. The \textit{type norm} of $(K,\Phi)$ is the multiplicative map
\begin{align*}
\N_{\Phi}:K&\rightarrow K^r,\\
x&\mapsto\prod_{\phi\in\Phi}\phi(x),
\end{align*}
satisfying $\N_{\Phi}(x)\overline{\N_{\Phi}(x)} = \N_{K/\QQ}(x)\in\QQ$.  
The type norm induces a homomorphism between the groups of fractional ideals~$I_K$ and $I_{K^r}$ by sending $\bb\in I_K$ to $\bb'\in I_{K^r}$ such that $\bb'\OO_{N}=\prod\limits_{\phi\in\Phi}\phi(\bb)\OO_{N}$  (Shimura \cite[Proposition 29]{shimura3}). 

An abelian variety $A$ over a field $k$ of characteristic $0$ of dimension $g$
has CM by an order $\OO_K$ if $K$ has degree~$2g$ and there is an
embedding $\theta:K\rightarrow \End(A_{\overline{k}})\otimes\QQ$.
Let $\Tgt_{0}(A)$ be the tangent space of $A$ over $\overline{k}$ at $0$.
Given $A$ with CM by $\OO_K$, let~$\Phi$ be the set of homomorphisms
$K \rightarrow \overline{k}$ occurring in the representation of $K$ in 
$\End_{\overline{k}}(\Tgt_{0}(A_{\overline{k}}))$. Then $\Phi$ is a CM type of~$K$, 
and we say that $(A,\theta)$ is of type $(K,\Phi)$. A polarized abelian 
variety of type $(K,\Phi)$ is a triple $P=(A,\theta, \varphi)$ formed by 
an abelian variety $(A,\theta)$ of type $(K,\Phi)$ and a polarization 
$\varphi$ of~$A$ such that $\theta(K)$ is stable under the involution 
of $\End_{\overline{k}}(A)\otimes\QQ$ determined by $\varphi$. 
For more details see Shimura \cite[Chapter 14]{shimura3}.

We say that an abelian variety is \textit{simple} if it is not isogenous over $\overline{k}$ to product of abelian varieties of lower dimension.

\begin{thm} [Lang \cite{lang}, I.3.5]\label{simple}
An \textit{abelian variety of type} $(K,\Phi)$ is simple if and only if $\Phi$ is primitive. \qed
\end{thm}

\begin{thm}[{First Main Theorem of Complex Multiplication, \cite[Main Theorem~1]{shimura3} in \S 15.3}] \label{main theorem} Let  $(K,\Phi)$ be a primitive CM type and let $(K^r,\Phi^r)$ be its reflex. Let $P=(A,\theta,\varphi)$ be a polarized simple abelian variety of type $(K,\Phi)$ with CM by~$\OO_K$. Let $M$ be the field of moduli of $(A,\varphi)$. Then $K^r\cdot M$ is the unramified class field over~$K^r$ corresponding to the ideal group  \[I_{0}(\Phi^{r}):=\{\bb \in I_{K^{r}}:\ \N_{\Phi^{r}}(\bb)=(\alpha)\ \text{and}\ \alpha\overline{\alpha}\in\QQ\ \text{\normalfont{for some}} \ \alpha\in K^{\times} \}.\]
\end{thm}

The \textit{Jacobian} $J(C)$ of a curve $C/k$ of genus $g$ is an abelian
variety of dimension~$g$ such that we have  $J(C)({\overline{k}}) = \Pic^{0}(C_{\overline{k}})$;
for details we refer to \cite{milnejac}.
We say that a curve has \emph{CM} if its Jacobian does.

\begin{cor} \label{cor:clnrone}
If a curve $C$ is of primitive type $(K,\Phi)$ with CM by $\OO_K$ and defined over $K^r$, then the CM class group $I_{K^r}/I_{0}(\Phi^{r})$ is trivial. \qed
\end{cor}
\begin{remk}
	The converse to Corollary~\ref{cor:clnrone} is true as well 
	and is due to Milne~\cite{milne-definition, milne-correction}.
	See also Bouyer-Streng \cite[Theorem~5.3]{bouyer}.
\end{remk}

In the genus-$2$ case, a quartic CM field $K$ is either cyclic Galois,
biquadratic Galois, or non-Galois with Galois group $D_4$
(Shimura \cite[Example 8.4(2)]{shimura3}).
We restrict ourselves to CM curves with a simple Jacobian,
which therefore have primitive CM types by Theorem~\ref{simple}.
In the biquadratic case, all CM types are non-primitive,
so we restrict to non-biquadratic fields.
The corresponding CM fields of such curves are not biquadratic,
by Example~8.4-(2) in Shimura \cite{shimura3}.

The \textit{CM class number one problem for CM fields of degree $2g$}
is the problem of finding all CM class number one pairs $(K,\Phi)$ of degree~$2g$. 
Theorems \ref{thm: non-normallist} and \ref{thm: cycliclist} solve
this problem for non-biquadratic quartic CM fields.

\section{A first bound}\label{non-normal-partone}

Sections~\ref{non-normal-partone}--\ref{sec:enumerating}
prove Theorem~\ref{thm: non-normallist}.
The case of cyclic CM fields uses the same ideas except for some shortcuts, and is treated in Section \ref{cyclic}.

In this section, we find an effective upper bound for the discriminant of the non-normal quartic CM 
fields with CM class number one.
The bound will be too large for practical purposes, hence we improve the bound
using ramification theory in Section~\ref{non-normal-parttwo}.

\subsection{An effective bound} \label{finiteness}

We first prove the following relation between the relative class 
number $h^{*}_K:=h_K/h_F$ and the number $t_K$ of ramified primes in $K/F$.

\begin{prop}\label{prop2} Let $K$ be a non-biquadratic quartic CM field and let $F$ be the real quadratic subfield of $K$. Suppose $I_{0}(\Phi^{r})=I_{K^r}$. Then we have $h^{*}_{K}=2^{t_K-1}$, where~$t_K$ is the number of primes in $F$ that are ramified in~$K$.

Moreover, we have $h^{*}_{K^r}=2^{t_{K^r}-1}$, where $t_{K^r}$ is the number of primes in $F^r$ that are ramified in~$K^r$.
\end{prop} 

\begin{remk}
In the case where $K/\QQ$ is cyclic quartic, this result is $(i)\Rightarrow (iii)$ of Proposition 4.5 in Murabayashi \cite{murab}. 
\end{remk}

On the other hand, if $K$ is a non-normal quartic CM field, Louboutin proves $h^{*}_{K}\approx \sqrt{d_K/d_F}$ with an effective error bound, see Proposition~\ref{louboutin lower bound}.  Putting this together with the result in Proposition~\ref{prop2} gives approximately $\sqrt{d_K/d_F}\leq 2^{t_K-1}$. As the left hand side grows more quickly than the right, this relation will give a bound on the discriminant, precisely see Proposition \ref{effectiveness}.

We begin the proof of Proposition \ref{prop2} with the following lemma. 

Let $I_K$ be the group of fractional ideals in $K$ and let $P_K$ be the group of principal fractional ideals in $K$. 

\begin{lem}\label{lem1} Let $K$ be a CM field and let $F$ be the maximal totally real subfield of $K$. Let $H$ denote the group $\Gal(K/F)$. Put $I^{H}_{K}=\{\bb\in I_{K}\ |\ 
\overline{\bb}=\bb\}$ and $P^{H}_K=P_K\cap I_K^{H}$. Suppose that the group 
of roots of unity $\mu_K$ of $K$ is $\{\pm1\}$. Then we have $h_{K}^{*}=2^{t_{K}-1}[I_K:I^{H}_{K}P_K]$, where~$t_K$ is the number of primes in $F$ that ramify in $K$.
\end{lem}

\begin{proof} We have the exact sequence 
\begin{equation}\label{exact}
1\rightarrow I_{F} \rightarrow I_{K}^{H} \rightarrow \bigoplus_{\p \ \text{prime of}\ F}\ZZ/e_{K/F}(\p)\ZZ\rightarrow 1
\end{equation}
and  
\[
\bigoplus\limits_{\p \ \text{prime of}\ F}\ZZ/e_{K/F}(\p)\ZZ\quad\cong\quad(\ZZ/2\ZZ)^{t_K}.\]   The map $\varphi:I_{K}^{H}\rightarrow I_{K}/P_{K}$ induces an isomorphism $$I_{K}^{H}/P_{K}^{H}\cong \text{im}(\varphi)=I^{H}_{K}P_K/P_K$$ so by (\ref{exact}), we have \[h_F = [I_F:P_F] = \frac{[I_{K}^{H}:P_{K}^{H}][P_{K}^{H}:P_{F}]}{[I_{K}^{H}:I_{F}]}=2^{-t_K}[I^{H}_{K}P_K:P_K][P_{K}^{H}:P_{F}],\] hence \[h^{*}_{K}:=\frac{h_{K}}{h_F}=2^{t_K}\frac{[I_K:I_{K}^{H}P_K]}{[P_{K}^{H}:P_{F}]}.
\]
It now suffices to prove $[P_{K}^{H}:P_{F}]=2$, for which we claim that there is a surjective group 
homomorphism ${\lambda}: P_{K}^{H}{\rightarrow} \mu_K = \{\pm1\}$ given by $ 
\lambda((\alpha))={\alpha}/{\overline{\alpha}}$ with kernel $P_F$. 

\textit{Proof of the Claim}.  The map $\lambda$ is well-defined because the image is independent of the 
choice of $\alpha$ as $\OO^{*}_K=\OO^{*}_F$ (cf. \cite[Lemma 1]{loub5}). It is clear that it is a 
homomorphism. Since ${\alpha}/{\overline{\alpha}}=1$ if and only if ${\alpha}\in F^\times$, we have $\ker(\lambda) = P_F$. As $K=F(\sqrt{-\beta})$ with a totally positive element $\beta$ in~$F$, 
we have $\lambda((\sqrt{-\beta}))=-1$, so $\lambda$ is surjective.  

Hence $[P_{K}^{H}:P_{F}]=|\im(\lambda)| = 2$. 
\end{proof}

\begin{lem} \label{lemext} 
Let $(K,\Phi)$ be a primitive CM pair. Then we have $[I_K:I^{H}_{K}P_K]\leq[I_{K^r}:I_{0}(\Phi^r)]$. Moreover, we have $[I_{K^r}: I^{H'}_{K^r}P_{K^r}]\leq[I_{K^r}:I_{0}(\Phi^r)]$, where $H'=\Gal(K^r/F^r)$.
\end{lem}

\begin{proof} To prove the first assertion, we show that the kernel of the map 
$\N_{\Phi}:I_K\rightarrow I_{K^r}/I_{0}(\Phi^r)$ is contained in $I^{H}_{K}P_K$.  For any $\aaa\in 
I_K$, we can compute (see \cite[(3.1)]{shimura2})
\[\N_{\Phi^{r}}\N_{\Phi}(\aaa)=\N_{K/\QQ}(\aaa)\frac{\aaa}{\overline{\aaa}}.\] 
Suppose $\N_{\Phi}(\aaa)\in I_{0}(\Phi^r)$. Then $N_{K/\QQ}(\aaa)\frac{\aaa}
{\overline{\aaa}}=(\alpha)$, where $\alpha\in K^{\times}$ and $\alpha\overline{\alpha}=N_{K^r/\QQ}
(N_{\Phi}(\aaa))=N_{K/\QQ}(\aaa)^2\in\QQ$. So $\frac{\aaa}{\overline{\aaa}}=(\beta)$, where 
$\beta=\N_{K/\QQ}(\aaa)^{-1}\cdot\alpha$, and hence $\beta\overline{\beta}=1$. There is a $\gamma\in 
K^{\times}$ such that $\beta=\frac{\overline{\gamma}}{\gamma}$ (this is a special case of Hilbert's 
Theorem 90, but can be seen directly by taking $\gamma = \overline{\epsilon} + 
\overline{\beta}\epsilon$ for any $\epsilon\in {K}$ with $\gamma\neq 0$). Thus we have 
$\aaa=\overline{\gamma\aaa}\cdot(\frac{1}{\gamma})\in  I^{H}_{K}P_{K}$ and therefore 
$[I_K:I^{H}_{K}P_K]\leq[I_{K^r}:I_{0}(\Phi^r)]$.

For the second assertion, we show $I_0(\Phi^r)\subset I_{K^r}^{H'}P_{K^r}$.
By \cite[(3.2)]{shimura2}, we have
\[\N_{\Phi}\N_{\Phi^r}(\bb)=\N_{K^r/\QQ}(\bb)\frac{\bb}{\overline{\bb}}.\]
Suppose $\bb\in I_{0}(\Phi^r)$. Then we have $N_{K^r/\QQ}(\bb)\frac{\bb}{\overline{\bb}}=(\alpha)$, where 
$\alpha\in {K^r}^{\times}$ and $\alpha\overline{\alpha}=N_{\Phi}(N_{K^r/\QQ}(\bb))$ $=N_{K^r/\QQ}
(\bb)^2\in\QQ$. We finish the proof of  $\bb\in  I^{H'}_{K^r}P_{K^r}$ exactly as above. 
\end{proof}

\begin{proof}[Proof of Proposition~\ref{prop2}]
	In case $K\cong\QQ(\zeta_5)$, we have $h_{K}^* = h_{K^r}^* = t_{K} = t_{K^r}=1$
	so the result follows. In all other cases,
	we have
$\mu_K=\{\pm 1\}$, so by Lemma \ref{lem1}, we have  $h^{*}_K=2^{t_{K}-1}
[I_K:I_{K}^{H}P_K]$. We showed in Lemma \ref{lemext} that, under the assumption $I_{0}
(\Phi^{r})=I_{K^{r}}$, the quotient $I_K/I_{K}^{H}P_K$ is trivial. Therefore, we have 
$h^{*}_{K}=2^{t_K-1}$.

Similarly, since $\mu_K=\{\pm 1\}$, by Lemma \ref{lem1}, we have  $h^{*}_{K^r}=2^{t_{K^r}-1}[I_{K^r}:I_{K^r}^{H^r}P_{K^r}]$. By Lemma \ref{lemext}, under the assumption $I_{0}(\Phi^{r})=I_{K^{r}}$, we have $[I_{K^r}:I_{K^r}^{H^r}P_{K^r}] = 1$, hence we get $h^{*}_{K^r}=2^{t_{K^r}-1}$.
\end{proof}

The next step is to use the following bound from analytic number theory. 

Let $d_M$ denote the discriminant of a number field $M$.

\begin{prop}[{Louboutin \cite{loub2}, Remark 27 (1)}] \label{louboutin lower bound} Let $N$ be the 
normal closure of a non-normal quartic CM field $K$ with Galois group $D_4$. Assume 
$d_{N}^{1/8}\geq222$. Then
\begin{equation}\label{rel.cl.2}\qquad \qquad \quad h^{*}_{K}\geq\frac{2\sqrt{d_{K}/d_{F}}}
{\sqrt{e}\pi^{2}(\log(d_{K}/d_{F})+0.057)^{2}}.  
\end{equation}  \qed
\end{prop}

\begin{prop} \label{effectiveness}
Let $K$ be a non-normal quartic CM field and let $F$ be the real quadratic subfield of $K$. Let $\Phi$ be a primitive CM type of $K$. Suppose  $I_{0}(\Phi^{r})=I_{K^{r}}$. Then we have 
$d_K/d_F\leq 2\cdot 2^{19}$. 
\end{prop}

\begin{proof} 
Let \[f(D) = \frac{2\sqrt{D}}{\sqrt{e}\pi^2(\log(D)+0.057)^2}\quad \text{and}\quad 
g(t)=2^{-t+1}f(5\Delta_{\lceil t/2\rceil}\Delta_{\lfloor t/2\rfloor}),\]
where $\Delta_t$ is the product of the first $t$ prime numbers,
and $\lceil\cdot\rceil$ and $\lfloor\cdot\rfloor$ round up and down
to integers.

Here, if $D = d_K/d_F$, then $f(D)$ is the right hand side
of the inequality (\ref{rel.cl.2}) in 
Proposition~\ref{louboutin lower bound}. 
The quotient $d_{K}/d_{F}^2$ is the norm of the relative
discriminant, which is divisible by the product of the norms of the primes of $F$
that are ramified in~$K/F$. Hence $d_{K}/d_{F}\geq d_F \Delta_{\lceil t_K/2\rceil}
\Delta_{\lfloor t_K/2\rfloor}$, where $t_K$ is the number of primes in~$F$ that are ramified in $K$.
We also have $d_F\geq 5$.

On the other hand, the function $f$ is monotonically increasing for $D>52$, so if $t_K\geq 3$ 
then $f(d_{K}/d_{F})\geq f(5\Delta_{\lceil t_K/2\rceil}
\Delta_{\lfloor t_K/2\rfloor})$. 
Therefore, under the assumption $I_{0}(\Phi^{r})=I_{K^{r}}$, by Proposition \ref {prop2}, we get (if $t_K\geq 3$)
\begin{equation}\label{lower bd. with delta}
{2^{t_K-1}}\geq f(d_{K}/d_{F})\geq f(5\Delta_{\lceil t_K/2\rceil}
\Delta_{\lfloor t_K/2\rfloor})
\end{equation}
and hence $g(t_K)\leq 1$. The function $g(t)$ is monotonically increasing for $t\geq 8$
and~$g(t) > 1$ if $t=20$.
Therefore, we get $t_K\leq 19$ and $h^{*}_{K} \leq 2^{18}$, hence $d_K/d_F< 2\cdot10^{19}$.
\end{proof}

The bound that we get in Proposition \ref{effectiveness} is unfortunately too large to
list all the fields.
In Section~\ref{non-normal-parttwo} we study ramification of primes in
$N/\QQ$ and find a sharper upper bound for $d_{K^r}/d_{F^r}$, see Proposition \ref{t<6}.

\section{Improved bounds}\label{non-normal-parttwo}

In this section, under the assumption  $I_{0}(\Phi^{r})=I_{K^{r}}$, we study the ramification behavior of primes in $N/\QQ$, and prove that almost all rational primes that are ramified in~$K^r/F^r$ are inert in~$F^r$. 
We precisely prove the following proposition.

\newcommand{\contentsofpropraminFr}{
	Let $K$ be a non-normal quartic CM field and let $F$ be its real quadratic subfield.
Let $\Phi$ be a primitive CM type of $K$
and suppose $I_{0}(\Phi^{r})=I_{K^r}$.

Then we have $F=\QQ(\sqrt{p})$ and $F^r=\QQ(\sqrt{q})$, where~$p$ and~$q$ are prime numbers with $q\not\equiv3\ (\mod\ 4)$.
The primes $p$ and $q$ are split in $F^r$ and $F$ respectively.
Moreover, all primes coprime to
$p$ and $q$
that are ramified in $K^r/F^r$  are inert in~$F/\QQ$ and~$F^r/\QQ$. 
}

\begin{prop}\label{prop: ram_in_Fr}
	\contentsofpropraminFr
\end{prop}

This proposition implies that $d_{K^r}/d_{F^r}$ grows as the square of the product
of such ramified primes and we get a lower bound on $f(d_{K^r}/d_{F^r})$
that grows much faster with $t_{K^r}$ than what we had in the proof
of Proposition~\ref{effectiveness}.
Hence we obtain a better upper bound on~$d_{K^r}/d_{F^r}$, see Proposition~\ref{t<6}.

We begin the proof of Proposition~\ref{prop: ram_in_Fr} with exploring
the ramification behavior of primes in $N/\QQ$.

 \subsection{Non-normal quartic CM fields} 
Suppose that $K/\QQ$ is a non-normal quartic CM field and $F$ is the real quadratic subfield of $K$. The 
normal closure $N$ is a dihedral CM field of degree $8$ with Galois group $G:=\Gal(N/\QQ)= \langle x,
\ y:y^{4}=x^{2}=(xy)^{2}=\id\rangle$. Complex conjugation $\overline{\cdot}$ is $y^2$ in this 
notation and the CM field $K$ is the subfield of $N$ fixed by~$\langle x\rangle$. Let $\Phi$ 
be a CM type of~$K$ with values in~$N'$. We can (and do) identify $N$ with a subfield 
of $N'$ in such a way that $\Phi=\{\id, y|_{K}\}$. Then the reflex field $K^r$ of~$\Phi$ is 
the fixed field of $\langle xy\rangle$, which is a non-normal quartic CM field non-isomorphic to $K$ 
with reflex type $\Phi^{r}=\{\id, y^{3}|_{K^r}\}$, (see \cite[Examples~8.4., 2(C)]{shimura3}). Denote 
the quadratic subfield of $K^{r}$ by $F^r$. 

\begin{figure}[h] 
     \begin{subfigure}
\centering
        \begin{tikzpicture}[scale=0.80] 
    \draw (0, -1.2) node[]{$N$}; 
    \draw (0, -3) node[]{$N_+$};
	\draw (-1.5, -3) node[]{$K$};
	\draw (-3, -3) node[]{$K'$};
	\draw (1.5, -3) node[]{$K^r$};
	\draw (3, -3) node[]{$K'^r$};
    \draw (-2.25, -4.5) node[]{$F$};
    \draw (2.25, -4.5) node[]{$F^r$};
	\draw (0, -4.5) node[]{$F_+$};
    \draw (0, -6) node[]{$\QQ$};
    \draw (0, -4) -- (0, -3.25); 
    \draw (0, -2.50) -- (0, -1.5); 
    \draw (-0.50, -3.25) -- (-2.0, -4.25); 
			\draw (0.50, -3.25) -- (2.0, -4.25); 
    \draw (-2.5, -4.2) -- (-3, -3.28); 
    \draw (2.5, -4.2) -- (3, -3.28); 
		\draw (-2.25, -4.2) -- (-1.5, -3.28); 
    \draw (2.25, -4.2) -- (1.5, -3.28); 
	\draw (-2.75, -2.75) -- (-0.25, -1.25); 
	\draw (2.75, -2.75) -- (0.25, -1.25); 
		\draw (-1.25, -2.75) -- (-0.15, -1.50); 
	\draw (1.25, -2.75) -- (0.15, -1.50); 
    \draw (-2.25, -4.75) -- (-0.50, -5.75); 
    \draw (2.25, -4.75) -- (0.50, -5.75); 
		\draw (0, -4.75) -- (0, -5.50); 
\end{tikzpicture}
\end{subfigure}
     \begin{subfigure}
\centering
        \begin{tikzpicture}[scale=0.80]
    \draw (0, -1.2) node[]{$1$};
    \draw (0, -3) node[]{$\langle y^2\rangle$};
	\draw (-1.5, -3) node[]{$\langle x\rangle$};
	\draw (-3, -3) node[]{$\langle xy^2\rangle$};
	\draw (1.5, -3) node[]{$\langle xy\rangle$};
	\draw (3, -3) node[]{$\langle xy^3\rangle$};
    \draw (-2.25, -4.5) node[]{$\langle x,\ y^2\rangle$};
    \draw (2.25, -4.5) node[]{$\langle xy,\ y^2\rangle$};
	\draw (0, -4.5) node[]{$\langle y\rangle$};
    \draw (0, -6) node[]{$G$};
    \draw (0, -4) -- (0, -3.25); 
    \draw (0, -2.50) -- (0, -1.5); 
    \draw (-0.50, -3.25) -- (-2.0, -4.25); 
			\draw (0.50, -3.25) -- (2.0, -4.25); 
    \draw (-2.5, -4.2) -- (-3, -3.28); 
    \draw (2.5, -4.2) -- (3, -3.28); 
		\draw (-2.25, -4.2) -- (-1.5, -3.28); 
    \draw (2.25, -4.2) -- (1.5, -3.28); 
	\draw (-2.75, -2.75) -- (-0.25, -1.25); 
	\draw (2.75, -2.75) -- (0.25, -1.25); 
		\draw (-1.25, -2.75) -- (-0.15, -1.50); 
	\draw (1.25, -2.75) -- (0.15, -1.50); 
    \draw (-2.25, -4.75) -- (-0.50, -5.75); 
    \draw (2.25, -4.75) -- (0.50, -5.75); 
		\draw (0, -4.75) -- (0, -5.50); 
\end{tikzpicture}
\end{subfigure}
\caption{Lattice of subfields and subgroups} \label{diagram1} 
 \end{figure}
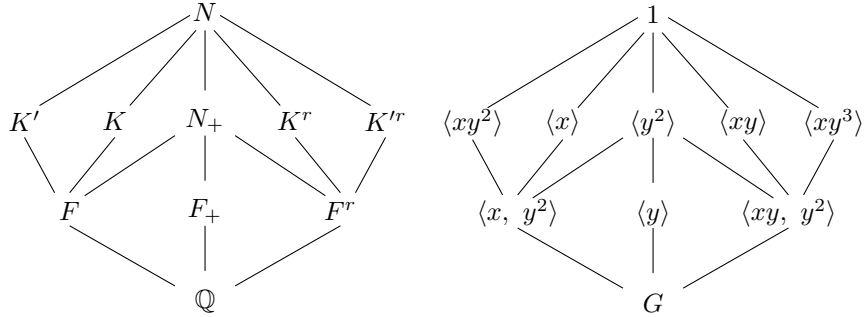

Let $N_+$ be the maximal totally real subfield of $N$, and let $F_+$ be the quadratic subfield of $N_+$ such that $N/F_+$ is cyclic.

\subsection{\texorpdfstring{Ramification of primes in $N/\QQ$}{Ramification of primes in N over Q}}

We will use the following well-known result.
\begin{lem}\label{inertia gp. not V4}
Let $M/L$ be a Galois extension of number fields and $\q$ be a prime
of~$M$ over an odd prime number.
Then there is no surjective homomorphism from a subgroup of $I_\q$ to a Klein four group $V_4$.
\end{lem}

\begin{proof} Suppose that there is a surjective homomorphism 
from a subgroup of $\I_\q$ to $V_4$. In other words, there exists a
biquadratic intermediate extension $E/F$ of $M/L$
such that $\p = \q\cap F$ is totally ramified in~$E/F$.
Denote the three quadratic intermediate 
extensions by $E_i = F(\sqrt{\alpha_i})$ for $i=1,\ 2,\ 3$. Without loss of generality, take $\alpha_i\in \mathcal{O}_F$ with
$\ord_\p(\alpha_i)\in\{0,1\}$ for each $i$. Note $\OO_{E_i}$ contains $\OO_F[\sqrt{\alpha_i}]$ of 
relative discriminant $4\alpha_i$ over $\OO_F$. Since $\p$ is odd, this implies that the relative 
discriminant $\Delta(E_i/F)$ of $\OO_{E_i}$ has $\ord_{\p}(\Delta(E_i/F)) = \ord_{\p}(\alpha_i)$. At the same time, we have $E_3 = F(\sqrt{\alpha_1\alpha_2})$ so $\p$ ramifies in $E_i$ for an even number of $i$'s.
Contradiction.
\end{proof}

\begin{lem} \label{ram1}
	Let $K$ be a non-normal quartic CM field and $\Phi$ a (primitive)
	CM type of~$K$. Let
	$K^r$ be the reflex field of $(K,\Phi)$, and let $F$ and $F^r$ be the real quadratic subfields of $K$ and~$K^r$.
	Then the following assertions hold.
\begin{itemize} 
\item[(i)]  If a prime $p$ is ramified in both $F$ and $F^r$, then it is totally ramified in $K/\QQ$ and $K^r/\QQ$. 
\item[(ii)] If an odd prime $p$ is ramified in $F$ (in $F^r$, respectively) as well as in $F_+$, then~$p$ splits in $F^r$ (in $F$, respectively). Moreover, at least one of the primes above $p$ in $F^r$ is ramified in $K^r/F^r$ (in $K/F$, respectively). 
\end{itemize}
\end{lem}

\begin{proof}   
The statements (i) and (ii) are clear from Table \ref{table} on page \pageref{table}. Alternatively, one can also prove the statements as follows: 
\begin{itemize}
\item[(i)] Let $\p_{N}$ be a prime of $N$ above $p$ that is ramified in both $F/\QQ$ and $F^r/\QQ$. Then the maximal unramified subextension of $N/\QQ$ is contained in $F_+$.   Therefore, the inertia group of $\p_{N}$ contains $\Gal(N/F_+)=\langle y \rangle$. By computing ramification indices in the diagram of subfields one by one, we see that the prime $p$ is totally ramified in $K$ and $K^r$. 

\item[(ii)] Let $p$ be an odd prime that is ramified in $F/\QQ$ and $F_+/\QQ$ and $\p_{N}$ be a prime above $p$ in $N$. The inertia group of an odd prime cannot be a biquadratic group by Lemma \ref{inertia gp. not V4}, so $\I_{\p_N}$ is a proper subgroup of $\Gal(N/F^r)$. Since~$\I_{\p_N}$ is a normal subgroup in $\D_{\p_N}$, the group $\D_{\p_N}$ cannot be the full Galois group $\Gal(N/\QQ)$. So~$\D_{\p_N}$ is a proper subgroup of $\Gal(N/F^r)$ and hence $p$ splits in $F^r$. Moreover, since~$p$ is ramified in $F$, hence in $K$,  hence in $K^r$, at least one of the primes above $p$ in $F^r$ is ramified in $K^r$. Since $F$ and $F^r$ are symmetric in $N/\QQ$, the same argument holds for $F^r$ as well. \end{itemize} \end{proof}

\begin{lem} \label{ram2} Let the notation be as in Lemma~\ref{ram1}. Assuming \linebreak $I_{0}(\Phi^{r})=I_{K^{r}}$, if $K^r$ has a prime $\p$ of prime norm $p$ with~$\overline{\p}=\p$, then we have $F=\QQ(\sqrt{p})$.
\end{lem}

\begin{proof} By the assumption,  we have
       \[\N_{\Phi^{r}}(\p)=(\alpha)\ \text{for some}\ \alpha\in K^\times \text{such that}\ \alpha\overline{\alpha}=\N_{K^r/\QQ}(\p)=p.\] 
Since $\overline{\p}=\p$, we have $(\alpha)=(\overline{\alpha})$, and so
$\alpha=\epsilon\overline{\alpha}$ for a unit $\epsilon$ in $\OO^{\times}_{K}$
with $\epsilon\overline{\epsilon}=1$, hence a root of unity.
Since $\mu_K=\{\pm1\}$, we get $\alpha^{2}=\pm p$.
The case $\alpha^{2}=-p$ is not possible, since $K$ has no
imaginary  quadratic intermediate field.
Hence we have~$\alpha^{2}=p$ and so $\sqrt{p}\in F$. 
\end{proof}

\begin{cor}\label{totram} The notation being as in Lemma~\ref{ram1}, suppose \linebreak$I_{0}(\Phi^{r})=I_{K^{r}}$.
If $p$ is totally ramified in $K^r/\QQ$, or splits in $F^r/\QQ$
and at least one of the primes over $p$ in $F^r$ ramifies in $K^r/F^r$,
then we have~$F=\QQ(\sqrt{p})$. \qed
\end{cor}

\begin{prop} \label{prop:sqrtp} Suppose $I_{0}(\Phi^{r})=I_{K^r}$. Then $F=\QQ(\sqrt{p})$, where $p$ is a rational prime.
\end{prop}

\begin{proof}
Suppose that there is an odd prime $p$ that is ramified in $F$. Then~$p$ is ramified either in $F$ and $F^r$ or in $F$ and $F_+$.

If $p$ is ramified in both $F$ and $F^r$, then by Lemma \ref{ram1}-(i), the prime~$p$ is totally ramified in~$K^r/\QQ$. If $p$ is ramified in $F$ and $F_+$, then by Lemma \ref{ram1}-(ii), the prime~$p$ splits in $F^r$ and at least one of the primes over $p$ in $F^r$ ramifies in $K^r/F^r$. In both cases, Corollary~\ref{totram} tells us that  $F=\QQ(\sqrt{p})$. 

Therefore, if an odd prime $p$ is ramified in $F$, then we have $F=\QQ(\sqrt{p})$. If no odd prime ramifies in $F$, then the only prime that ramifies in $F$ is~$2$ so we have~$F=\QQ(\sqrt{2})$.
\end{proof}

\begin{lem}\label{ram3} Suppose $I_{0}(\Phi^{r})=I_{K^{r}}$. Then the following assertions are true. 
\begin{itemize}
\item[(i)] If a rational prime $l$ is unramified in both $F/\QQ$ and $F^r/\QQ$, but is ramified in~$K/\QQ$ or $K^r/\QQ$, then all primes above $l$ in $F$ and $F^r$ are ramified in $K/F$ and $K^{r}/F^{r}$ and $l$ is inert in $F^r$. 
	
\item[(ii)] If $F=\QQ(\sqrt{p})$ with a prime number $p\equiv 3\ (\mod\ 4)$, then $2$ is inert in $F^r$.
\end{itemize} 
\end{lem}
\begin{proof}
(i). 
Under the assumption $I_{0}(\Phi^{r})=I_{K^{r}}$
(see the final column of Table~\ref{table}),
the only possible decomposition types for a prime that is
ramified in $K$ and unramified in both $F$ and $F^r$ are (2) and (3) in  Table \ref{table}.
Hence the statement follows. 

(ii).
The prime $2$ is ramified in $F$ since $p\equiv 3\ (\mod\ 4)$.
By Table~\ref{table}, we see that if $2$ is ramified or split in $F^r$, then we have $\sqrt{2}\in F$, a contradiction. This implies that $2$ is inert in $F^r$.
\end{proof}

\newpage \label{key}
\begin{landscape}
\begin{table}[h]
\caption{Ramification table of a non-normal quartic CM field}\label{table} 
\begin{center} 
\resizebox{\columnwidth}{!}{
\renewcommand{\arraystretch}{1.6}
\begin{tabular}{|c|c|c|c|c|c|c|c|c|c|c|}  
\hline
 Case & $\I$ & $\D$ & {decomp. of $p$ in $N$} & {decomp. of $p$ in $K$} & {decomp. of $p$ in $F$} & {decomp. of $p$ in $F_+$} &{decomp. of $p$ in $F^r$} & {decomp. of $p$ in $K^r$} & $N_{\Phi^r}(\p_{K^{r},1})$ & $\sqrt{p}\in F$  \\  
 \noalign{\hrule height 2pt} 
(1)* &  $\langle y^2\rangle$ 
& $\langle y^2\rangle$ 
& $\p^{2}_{N,1} \p^{2}_{N,x} \p^{2}_{N,y} \p^{2}_{N,xy}$ 
&$\p^{2}_{K,1} \p^{2}_{K,y}$
&$\p_{F,1} \p_{F,y}$
&$\p_{F_+,1} \p_{F_+,y}$
&$\p_{F^{r},1} \p_{F^{r},y}$
&$\p^{2}_{K^{r},1} \p^{2}_{K^{r},y}$
& $\p_{K,1} \p_{K,y}$
& \checkmark \\
\hline 
(2) & $\langle y^2\rangle$ 
& $\langle y\rangle$ 
& $\p^{2}_{N,1} \p^{2}_{N,y}$ 
&$\p^{2}_{K,1}$
&$\p_{F,1}$
&$\p_{F_+,1} \p_{F_+,y}$
&$\p_{F^{r},1}$
&$\p^{2}_{K^{r},1}$
&$p$ 
&  \\ 
\hline 
(3) & $\langle y^2\rangle$ 
& $\langle x,\ y^2\rangle$ 
& $\p^{2}_{N,1} \p^{2}_{N,y}$ 
&$\p^{2}_{K,1} \p^{2}_{K,y}$ 
&$\p_{F,1} \p_{F, y}$ 
&$\p_{F_+,1}$
&$\p_{F^{r},1}$
&$\p^{2}_{K^{r},1}$
&$p$
& \\ 
 \hline 
(4)* & $\langle y^2\rangle$ 
&$\langle xy,\ y^2\rangle$ 
&$\p^{2}_{N,1} \p^{2}_{N,y}$ 
&$\p^{2}_{K,1}$
&$\p_{F,1}$
&$\p_{F_+,1}$
&$\p_{F^{r},1} \p_{F^{r},y}$
&$\p^{2}_{K^{r},1} \p^{2}_{K^{r},y}$
&$\p_{K,1}$
& \checkmark\\ 
\hline 
\multirow{2}{*}{(5)} 
\,(a) & $\langle x\rangle$ 
& $\langle x\rangle$ 
& $\p^{2}_{N,1} \p^{2}_{N,y}\p^{2}_{N,y^2}\p^{2}_{N,y^3}$ 
& $\p_{K,1}\p^{2}_{K,y}\p_{K,y^2}$
& $\p_{F,1}\p_{F,y}$
&$\p_{F_+,1}^2$
& $\p^{2}_{F^{r},1}$
& $\p^{2}_{K^{r},1}\p^{2}_{K^{r},y^2}$
&$\p_{K,1}\p_{K,y}$
& \\ 
\cdashline{2-11} 
\quad \quad (b) & $\langle xy^2\rangle$ 
& $\langle xy^2\rangle$ 
& $\p^{2}_{N,1}\p^{2}_{N,y}\p^{2}_{N,y^2}\p^{2}_{N,y^3}$ 
& $\p^{2}_{K,1}\p_{K,y}\p_{K,y^3}$
& $\p_{F,1}\p_{F,y}$
&$\p_{F_+,1}^2$
& $\p^{2}_{F^{r},1}$
& $\p^{2}_{K^{r},1}\p^{2}_{K^{r},y}$
&$\p_{K,1}\p_{K,y^3}$
&  \\ 
\hline 
\multirow{2}{*}{(6)}
\,(a) & $\langle x\rangle$ 
& $\langle x,\ y^2\rangle$ 
&$\p^{2}_{N,1}\p^{2}_{N,y}$ 
&$\p_{K,1}\p^{2}_{K,y}$
&$\p_{F,1}\p_{F,y}$
&$\p_{F_+,1}^2$
&$\p^{2}_{F^{r},1}$
&$\p^{2}_{K^{r},1}$
&$p$
&  \\ 
\cdashline{2-11} 
\quad \quad (b) & $\langle xy^2\rangle$ 
&$\langle x, y^2\rangle$ 
&$\p^{2}_{N,1}\p^{2}_{N,y}$ 
&$\p^{2}_{K,1}\p_{K,y}$
&$\p_{F,1}\p_{F,y}$
&$\p_{F_+,1}^2$
&$\p^{2}_{F^{r},1}$
&$\p^{2}_{K^{r},1}$
&$p$
& \\ 
\hline 
\multirow{2}{*}{(7)}
\,(a) & $\langle xy\rangle$ 
& $\langle xy\rangle$ 
& $\p^{2}_{N,1}\p^{2}_{N,y}\p^{2}_{N,y^2}\p^{2}_{N,y^3}$ 
& $\p^{2}_{K,1}\p^{2}_{K,y^3}$
& $\p^{2}_{F,1}$
&$\p_{F_+,1}^2$
& $\p_{F^{r},1}\p_{F^{r},y}$
& $\p^{2}_{K^{r},1}\p_{K^{r},y}\p_{K^{r},y^3}$
&$\p_{K,1}\p_{K,y^3}$
& \checkmark\\ 
\cdashline{2-11}
\quad \quad (b) & $\langle xy^3\rangle$ 
& $\langle xy^3\rangle$ 
& $\p^{2}_{N,1}\p^{2}_{N,y}\p^{2}_{N,y^2}\p^{2}_{N,y^3}$ 
& $\p^{2}_{K,1}\p^{2}_{K,y}$
& $\p^{2}_{F,1}$
&$\p_{F_+,1}^2$
& $\p_{F^{r},1}\p_{F^{r},y}$
& $\p_{K^{r},1}\p^{2}_{K^{r},y}\p_{K^{r},y^2}$
& $\p^{2}_{K,1}$
& \checkmark\\ 
\hline 
\multirow{2}{*}{(8)}
\,(a) & $\langle xy\rangle$ 
& $\langle xy, y^2\rangle$ 
&$\p^{2}_{N,1}\p^{2}_{N,y}$ 
&$\p^{2}_{K,1}$
&$\p^{2}_{F,1}$
&$\p_{F_+,1}^2$
&$\p_{F^{r},1}\p_{F^{r},y}$
&$\p^{2}_{K^{r},1}\p_{K^{r},y}$
&$\p_{K,1}$
& \checkmark \\ 
\cdashline{2-11}
\quad \quad (b) &$\langle xy^3\rangle$ 
&$\langle xy, y^2\rangle$ 
&$\p^{2}_{N,1}\p^{2}_{N,y}$ 
&$\p^{2}_{K,1}$
&$\p^{2}_{F,1}$
&$\p_{F_+,1}^2$
&$\p_{F^{r},1}\p_{F^{r},y}$
&$\p_{K^{r},1}\p^{2}_{K^{r},y}$
&$p$
& \checkmark \\ 
\hline 
(9) & $\langle y\rangle$ 
&$\langle y\rangle$ 
&$\p^{4}_{N,1}\p^{4}_{N,y}$ 
&$\p^{4}_{K,1}$
&$\p^{2}_{F,1}$
&$\p_{F_+,1}\p_{F_+,y}$
&$\p^{2}_{F^{r},1}$
&$\p^{4}_{K^{r},1}$
&$\p^{2}_{K,1}$
& \checkmark \\ 
\hline 
(10) & $\langle y\rangle$ 
& $G$ 
&$\p^{4}_{N,1}$ 
&$\p^{4}_{K,1}$
&$\p^{2}_{F,1}$
&$\p_{F_+,1}$
&$\p^{2}_{F^{r},1}$
&$\p^{4}_{K^{r},1}$
&$\p^{2}_{K,1}$
& \checkmark \\ 
\hline 
(11)* & $\langle x, y^2\rangle$ 
& $\langle x, y^2\rangle$ 
&$\p^{4}_{N,1}\p^{4}_{N,y}$ 
&$\p^{2}_{K,1}\p^{2}_{K,y}$
&$\p_{F,1}\p_{F,y}$
&$\p^{2}_{F_+,1}$
&$\p^{2}_{F^{r},1}$
&$\p^{4}_{K^{r},1}$
&$\p_{K,1}\p_{K,y}$
& \checkmark \\ 
\hline 
(12)* & $\langle x, y^2\rangle$ 
& $G$ 
&$\p^{4}_{N,1}$ 
&$\p^{2}_{K,1}$
&$\p_{F,1}$
&$\p_{F_+,1}^2$
&$\p^{2}_{F^{r},1}$
&$\p^{4}_{K^{r},1}$
&$\p_{K,1}$
& \checkmark \\ 
\hline 
(13) & $\langle xy, y^2\rangle$ 
& $\langle xy, y^2\rangle$ 
&$\p^{4}_{N,1}\p^{4}_{N,y}$ 
&$\p^{4}_{K,1}$
&$\p^{2}_{F,1}$
&$\p^{2}_{F_+,1}$
&$\p_{F^{r},1}\p_{F^{r},y}$
&$\p^{2}_{K^{r},1}\p^{2}_{K^{r},y}$
&$\p^{2}_{K,1}$
& \checkmark \\
\hline 
(14) & $\langle xy,y^2\rangle$ 
& $G$ 
&$\p^{4}_{N,1}$ 
&$\p^{4}_{K,1}$
&$\p^{2}_{F,1}$
&$\p_{F_+,1}^2$
&$\p_{F^{r},1}$
&$\p^{2}_{K^{r},1}$
&$p$
&  \\ 
\hline 
(15) & $G$ 
& $G$ 
&$\p^{8}_{N,1}$ 
&$\p^{4}_{K,1}$
&$\p^{2}_{F,1}$
&$\p^{2}_{F_+,1}$
&$\p^{2}_{F^{r},1}$
&$\p^{4}_{K^{r},1}$
&$\p^{2}_{K,1}$
& \checkmark \\ 
\hline
\end{tabular}
}
\end{center}
\mbox{}

\caption*{This table lists all $19$ pairs $(\I,\D)$ where $1\neq \I \triangleleft \D \leq D_4 = \langle x,\ y \rangle$ and $\D/\I$ is cyclic, partitioned into $15$ conjugacy classes (1) -- (15). In particular, it contains all possible inertia and decomposition groups of ramified primes of $N$. This table is a corrected subset of \cite[Table 3.5.1]{goren}. We restricted to $\I\neq 1$, added the case 8-(b), which is missing in \cite[Table 3.5.1]{goren}, and corrected the type norm column of some cases. The cases (11) -- (15) can only occur for the prime $2$, see Lemma \ref{inertia gp. not V4}. If there is a checkmark in the last column, then by Lemma~\ref{ram2}, such splitting implies $\sqrt{p}\in F$ (i.e., $F=\QQ(\sqrt{p})$) under the assumption $I_{0}(\Phi^{r})=I_{K^{r}}$. The cases with~* do not occur under the assumption $I_{0}(\Phi^{r})=I_{K^{r}}$ because $p$ is not ramified in $F$ in these cases, but on the other hand $\sqrt{p}\in F$ by Lemma \ref{ram2}.}
\end{table}
\end{landscape}

\subsection{\texorpdfstring{\textbf{Equality of $t_K$ and $t_{K^r}$}}{Equality of tK and tKr}}

In the previous section, we proved that the primes that are unramified in~$F$ and~$F^r$, but are ramified in $K^r/F^r$ are inert in $F^r$. Thus these primes contribute to $t_{K^r}$ (the number of primes in $F^r$ that are ramified in~$K^r$) with one prime, on the other hand they contribute to $t_K$ (the number of primes in $F$ that are ramified in $K$) with at least one prime and exactly two if the prime splits in $F/\QQ$. So if we could prove $t_K=t_{K^r}$, then that would approximately say that all such primes are inert in both $F$ and $F^r$.  

\begin{prop}[{Shimura, \cite[Proposition A.7.]{shimura2}}]\label{rel.cl.eq.}
	Let $K$ be a non-normal quartic CM field and $\Phi$ a (primitive)
CM type of~$K$. Let
$K^r$ be the reflex field of $(K,\Phi)$.
	Then we have $h^{*}_K = h^{*}_{K^r}$.
\end{prop}
\begin{proof} The idea of the proof is to first show \begin{equation}\label{zeta}\zeta_K(s)/\zeta_F(s)=\zeta_{K^r}(s)/\zeta_{F^r}(s)\end{equation} and then use the analytic class number formula at $s=0$.  Louboutin~\cite[Theorem~A]{loub5} shows the equality (\ref{zeta}) by writing the Dedekind zeta functions of $K$,~$K^r$,~$F$ and $F^r$ as a product of Artin $L$-functions and finding relations between these combinations of $L$-functions (see \cite[Theorem~A]{loub5}). 

We can also get this equality by comparing the local factors of the Euler products of the Dedekind $\zeta$-functions of the fields. By Table~\ref{table}, we see that each ramified prime in $N/\QQ$ has the same factors in the Euler products of the quotients of the Dedekind $\zeta$-functions on both sides of (\ref{zeta}). As an example, we take a rational prime~$p$ with ramification type (6 a) in Table \ref{table}, where the local factors for $p$ of the Dedekind $\zeta$-functions are as follows:
\begin {align*}\zeta_K(s)_{p}&= \frac{1}{1-\N{\p_{K,1}}^{-s}}\cdot\frac{1}{1-\N{\p_{K,y}}^{-s}}=\frac{1}{1-(p^2)^{-s}}\cdot \frac{1}{1-{p}^{-s}},\\
\zeta_F(s)_{p}&=  \frac{1}{1-\N{\p_{F,1}}^{-s}}\cdot \frac{1}{1-\N{\p_{F,y}}^{-s}}=\bigg ( \frac{1}{1-{p}^{-s}}\bigg )^2,\\
\zeta_{K^r}(s)_{p}&= \frac{1}{1-\N{\p_{K^r,1}}^{-s}}=\frac{1}{1-(p^2)^{-s}},\\
\zeta_{F^r}(s)_{p}&= \frac{1}{1-\N{\p_{K^r,1}}^{-s}}= \frac{1}{1-{p}^{-s}}.
\end{align*}
So for such a prime, we get \[\zeta_K(s)_{p}/\zeta_F(s)_{p} = \frac{1}{1+{p}^{-s}} = \zeta_{K^r}(s)_{p}/\zeta_{F^r}(s)_{p}.\]
Similarly, by using Table 3.5.1 in \cite{goren}, we can get this equality for the unramified primes as well. 
This gives an alternative proof of~\eqref{zeta}.

The analytic class number formula at $s = 0$ (see, \cite[Chapter 4]{wash}) says that the Dedekind zeta function $\zeta_M(s)$ of an algebraic number field~$M$ has a zero at $s=0$ and the derivative of $\zeta_M(s)$ at $s=0$ has the value \label{analyticclassnumber}
\[-\frac{h_M\cdot R_M}{\#\mu_M},\]
where $h_M$ is the class number, $R_M$ is the regulator, and
$\mu_M$ is the group of roots of unity.  

Since $\#\mu_K=2=\#\mu_F$ and $R_K=2R_F$ (see Washington, \cite[Proposition 4.16]{wash}), the analytic class number formula at $s = 0$ gives
\begin{align*} \lim_{s\rightarrow 0} \frac{ \zeta_K(s)}{\zeta_F(s)} =2h^*_{K}.
\end{align*}
Therefore, the equality of $h_K^*$ and $h^*_{K^r}$ follows from the identity (\ref{zeta}).
\end{proof}
Recall that $t_K$ (respectively $t_{K^r}$)
is the number of primes
of $F$ (respectively $F^r$) that ramify in $K$ (respectively $K^r$).
\begin{cor}\label{equal} In the situation of Proposition~\ref{rel.cl.eq.}, if
	 $I_{0}(\Phi^{r})=I_{K^{r}}$, then we have $t_{K}=t_{K^r}$. 
\end{cor}
\begin{proof} By Proposition \ref{prop2}, we have $h^{*}_{K}=2^{t_K-1}$ and $h^{*}_{K^r}=2^{t_{K^r}-1}$. Then by Proposition \ref{rel.cl.eq.}, we get $t_{K}=t_{K^r}$.
\end{proof}

\subsection{Proof of Proposition \ref{prop: ram_in_Fr}}

\begin{repprop}{prop: ram_in_Fr}
	\contentsofpropraminFr
\end{repprop}

\begin{proof}
	Recall from Proposition~\ref{prop:sqrtp} that $F = \QQ(\sqrt{p})$ for a prime number~$p$.
	If $p\not\equiv 3\ (\mathrm{mod}\ 4)$, then $p$ is the only prime that is ramified in $F$.
	If~$p\equiv 3\ (\mathrm{mod}\ 4)$, then $p$ and $2$ are two distinct primes that are ramified in~$F$,
	and hence by the final column of Table~\ref{table}, we get that $2$ is of ``Case (14)'', hence inert in~$F^r$.
	This shows that there are four types of prime numbers that ramify in $N/\QQ$:
\begin{itemize}
\item[(I)] The prime $p$, which is ramified in $F$ and possibly in $F^r$.
\item[(II)] The primes that are unramified in $F$, but ramified in $F^r$, say $q_1,\ \ldots,\ q_s$. 
\item[(III)] The primes that are unramified in $F$ and $F^r$, but ramified in $K$,  say $r_1,\ \ldots,\ r_m$. 
\item[(IV)] If $p\equiv\, 3\ (\mod\, 4)$, then $2\neq p$ is ramified in $F$ and inert in $F^r$
and is of ``Case (14)'' in Table~\ref{table}.
\end{itemize}

Next, we compute the contribution of each ramification type to $t_K$
and $t_{K^r}$.
Let~$f_{p}$ and $f^{r}_{p}$ be the contributions of the primes over $p$. 
Set $i_2=1$ if $p \equiv\, 3\ (\mod\ 4)$, and $i_2=0$ if $p \not\equiv\, 3\ (\mod\ 4)$. 

\noindent\textit{Claim.} We have $t_{K^r}=f^{r}_{p}+m+i_2$
and we have $t_K\geq f_{p}+ s+ m + i_2$ with equality only
if all primes of type (III) are inert in~$F$.

\noindent\textit{Proof of the claim.}
The contributions of $p$ (type (I)) are $f_p$ and $f_p^r$ by definition.\\
(II) By Table \ref{table} including Lemma \ref{ram2},
we see that 
for $i=1,\ldots,s$ the prime $q_i$ splits in $F$ and \textit{exactly} one of the primes above $q_i$ in $F$
 ramifies in $K/F$ and the unique prime above~$q_i$ in $F^r$ does not ramify in $K^r/F^r$.\\
 (III) By \mbox{Lemma \ref{ram3}-(i)}, we see that for $j=1,\ldots, m$ the prime $r_j$ is inert
 in $F^r$ and all primes over $r_j$ ramify in $K^r/F^r$ and~$K/F$.
 It follows that $r_j$
 contributes with \textit{exactly} one prime to $t_{K^r}$, and with \textit{at least}
 one prime to $t_K$ and with exactly one if and only if $r_j$ is inert in $F/\QQ$.\\
 (IV) If $p\equiv3\ (\mod\ 4)$, then by ``Case (14)'' of Table~\ref{table},
 we see that $2$ contributes exactly $1$ to $t_K$ and~$t_{K^r}$.\\
 So we get $t_K\geq f_{p}+ s+ m + i_2$ with equality if and only if all
 primes of type (III) are inert in $F$ and $t_{K^r}=f^{r}_{p}+m+i_2$,
 exactly as claimed.

Corollary \ref{equal} gives $t_K=t_{K^r}$, which by the claim gives $f_p^r - f_p \geq s$ with equality
if and only if all primes of type (III) are inert in~$F$.

We observe that $s\geq 1$ holds. Indeed, if $s=0$, then all primes that ramify in~$F^r$ also ramify in $F$. Hence $d_{F^r}$ divides $d_F$, which is equal to~$p$ if $p\equiv 1\ (\mod\ 4)$ and~$4p$ otherwise. So $F^r\cong F$, a contradiction. 

By Table~\ref{table}, we see $1\geq f^{r}_{p}-f_{p}$ with equality if and only if $p$ splits in~$F^r$.

Combining the three previous paragraphs, we get $f_p^r-f_p = s = 1$, all primes of type (III) are inert in~$F$
and $p$ splits in $F^r$.
As $F^r$ has a unique ramified prime~$q=q_1$, it is $F^r = \QQ(\sqrt{q})$ with $q\not\equiv 3\ (\mathrm{mod}\ 4)$.
And as mentioned in the proof of the claim, the prime $q=q_1$ splits in~$F$.
\end{proof}

\subsection{\texorpdfstring{A sharper bound for $d_{K^r}/d_{F^r}$}{A sharper bound}}\label{betterupper}
\begin{prop}\label{t<6} Let $K$ be a non-normal quartic CM field and let $F$ be its real quadratic subfield. Let $\Phi$ be a primitive CM type of $K$. Suppose $I_{0}(\Phi^{r})=I_{K^{r}}$ and~$d_{N}^{1/8}\geq222$. Then we have $h_{K^r}^{*} \leq 2^{5}$ and $d_{K^r}/d_{F^r}\leq3\cdot10^{10}$. 
\end{prop} 
\begin{proof}

Under the assumption $I_{0}(\Phi^{r})=I_{K^{r}}$, in Propositions \ref{prop:sqrtp} and \ref{prop: ram_in_Fr} we proved $F=\QQ(\sqrt{p})$ and $F^r=\QQ(\sqrt{q})$, where $p$ and $q$ are prime numbers. Additionally, we proved that at least one of the ramified primes above $p$ in $F^r$ is ramified in $K^r/F^r$, and the other ramified primes in $K^r/F^r$ are inert in $F^r$, say $r_1,\ldots,r_{t_{K^r}-1}$.  Therefore, we have $d_{K^r}/d_{F^r}\geq pqr^2_1\cdots r^2_{t_{K^r}-1}$.

Let
\[f(D) = \frac{2\sqrt{D}}{\sqrt{e}\pi^2(\log(D)+0.057)^2}\quad \text{and}\quad  g(t)=2^{-t+1}f(p_{t}p_{t+1}\Delta^2_{t-1}),\]
where $p_j$ is the $j$-th prime number and $\Delta_k=\prod^{k}_{j=1}p_j$. If $D = d_{K^r}/d_{F^r}$,
then we have $h_{K^r} \geq f(D)$ by Proposition \ref{louboutin lower bound}. 

Recall that, by the proof of Proposition \ref{effectiveness}, the function $f$ is monotonically increasing for $D>52$. Therefore,  if $t_{K^r}>3$, then we have $f(d_{K^r}/d_{F^r})>f(p_{t_{K^r}}p_{t_{K^r}+1}\Delta^2_{t_{K^r}-1})$. So in that case by Proposition \ref{prop2} and Corollary \ref{equal}, we have $h^{*}_{K^r}=2^{t_{K^r}-1}$, hence we get $g(t_{K^r})\leq 1$. Further, the function $g$ is monotonically increasing for $t\geq 4$ and $g(t)>1$ if $t=7$. So we get $t_{K^r}\leq 6$ and $h^{*}_{K^r}\leq 2^5$, hence $d_{K^r}/d_{F^r}\leq3\cdot10^{10}$. 
\end{proof}

\section{Enumerating the fields}\label{sec:enumerating}

To specify quartic CM fields, we use the following notation of the ECHIDNA database \cite{echidna}. Given a quartic CM field $K$, let~$D$ be the discriminant of the real quadratic subfield $F$ of $K$. Write $K=F(\sqrt{-\alpha})$ where $\alpha$ is a totally positive element 
of~$\OO_F$ and take~$\alpha$ such that $A:=\Tr_{F/\QQ}(\alpha)>0$ is 
minimal and let $B:=\N_{F/\QQ}(-\alpha)$. We choose $\alpha$ with minimal $B$ 
if there is more than one $B$ with the same~$A$. We use the triple $[D,A,B]$ to 
uniquely represent the isomorphism class of the CM field $K\cong\QQ[X]/(X^4 + AX^2 + B)$.
\begin{thm}\label{thm: non-normallist2} There exist exactly $63$ isomorphism classes of CM class number one non-normal quartic CM fields.  The fields are given by $K\cong\QQ[X]/(X^4 + AX^2 + B)\supset \QQ(\sqrt{D})$ where $[D,A,B]$ ranges over
\begin{align} \nonumber &[5,13,41],\ [5,17,61],\ [5,21,109],\ [5,26,149],\ [5,34,269],\ [5,41,389],\\ \nonumber & [8,10,17],\ [8,18,73],\ [8,22,89],\ [8,34,281],\ [8,38,233],\ [13,9,17],\\ \nonumber &  [13,18,29],\ [13,29,181],\ [13,41,157],\ [17,5,2],\ [17,15,52],\ [17,46,257],\\ \nonumber & [17,47,548],\ [29,9,13],\  [29,26,53],\ [41,11,20],\ [53,13,29],\ [61,9,5],\\ \nonumber & [73,9,2],\ [73,47,388],\ [89,11,8],\ [97,94,657],\ [109,17,45],\\ \nonumber &  [137,35,272],\ [149,13,5],\ [157,25,117],\ [181,41,13],\ [233,19,32],\\ \nonumber & \ [269,17,5],\ [281,17,2],\ [389,37,245] 
\end{align} with class number $1$,
\begin{align} \nonumber & [5,11,29],\ [5,33,261],\ [5,66,909],\ [8,50,425],\ [8,66,1017],\ [17,25,50],\\ \nonumber & [29,7,5],\ [29,21,45],\ [101,33,45],\ [113,33,18],\ [8,14,41],\ [8,26,137],\\ \nonumber & [12,8,13],\ [12,10,13],\ [12,14,37],\ [12,26,61],\ [12,26,157],\ [44,8,5],\\ \nonumber & [44,14,5],\ [76,18,5],\ [172,34,117],\ [236,32,20] 
\end{align} with class number $2$,

\quad\ \ $[257,23,68]$\\ with class number $3$, and

\quad \ \ $[8,30,153],\ [12,50,325],\ [44,42,45]$\\ with class number $4$.
\end{thm}

We begin the proof by combining the ramification results into the following explicit form for~$K^r$.

\begin{prop}\label{construct.1}
	Let $K$ be a non-normal quartic CM field and let $\Phi$ be a (primitive) CM type of $K$.
	Suppose $I_{0}(\Phi^{r})=I_{K^{r}}$.
	Then there exist prime numbers $p$, $q$, and $s_1< \cdots<s_{u}$ with
	$u\in\{t_{K^r}-1,\ t_{K^r}-2\}$ such that all of the following hold.
	
	We have $F=\QQ(\sqrt{p})$ and $F^r=\QQ(\sqrt{q})$ with $q\not\equiv 3\ (\mod\ 4)$.
	The primes $p$ and~$q$ are split in $F^r$ and $F$ respectively,
	and all primes $s_i$ are inert in both $F$ and $F^r$.
	There exists a prime $\p$ lying above $p$ in $F^r$ that ramifies in $K^r$,
	an odd $j\in\ZZ_{>0}$ and a totally positive generator $\pi$ of~$\p^j$.
	
	Moreover, for each such $\p$, 
	$\pi$ and~$j$, we have $K^r\cong\QQ(\sqrt{-\pi s_1\cdots s_{u}})$.
\end{prop}

\begin{proof}
 By Proposition~\ref{prop: ram_in_Fr}, we have $F=\QQ(\sqrt{p})$ and
 $F^r=\QQ(\sqrt{q})$, where $p$ and~$q$ are prime numbers with 
 $q\not\equiv3\ (\mod\ 4)$, and such that $p$ (respectively $q$)
 splits in~$\QQ(\sqrt{q})$ (respectively $\QQ(\sqrt{p})$).

There exists a totally positive element $\beta$ in $({F^r})^\times$ such that 
$K^r=F^r(\sqrt{-\beta})$, \mbox{where $\beta$} is uniquely defined \mbox{up to 
$(({F^r})^\times)^2$}, so without loss of generality, we can take $\beta$ in 
$\OO_{F^r}$.

Since $\OO_{K^r} \supset \OO_{F^r}[\sqrt{-\beta}] \supset \OO_{F^r}$, the quotient 
of the discriminant ideals 
$$\Delta(\OO_{K^r}/\OO_{F^r})/\Delta(\OO_{F^r}[\sqrt{-\beta}]/\OO_{F^r}) = \Delta(\OO_{K^r}/\OO_{F^r})/(-4\beta)$$ is a square ideal in $\OO_{F^r}$ 
(see Cohen \cite[pp.79]{cohen}). As $\beta$ is unique up to squares, and we can 
\mbox{take $\lp$-minimal} $\beta'\in\beta(F^\times)^2$ for each prime $\lp$ of 
$\OO_{F^r}$, we get  
\begin{equation}\label{order} \ord_{\lp}((\beta))\equiv \begin{cases} 
1\ (\mod\ 2)&\quad \text{if $\lp$ is ramified in $K^r/F^r $ and $\lp\nmid 2$,}\\
0\ (\mod\ 2) &\quad \text{if $\lp$ is not ramified in $K^r/F^r $,}\\
0\ \text{or}\ 1\ (\mod\ 2)&\quad \text{if $\lp$ is ramified in $K^r/F^r $ and $\lp|2$.}
\end{cases} 
\end{equation}

Let $\lp_1,\ldots,\lp_{t_{K^r}}\subseteq \OO_{F^r}$ be the primes
that ramify in $K^r/F^r$,
and let $l_i\in\ZZ_{>0}$ be the prime number in $\lp_i$.
Let 
$n_i>0$ be minimal such that $\lp_i^{n_i}$ is generated by a totally positive 
$\lambda_i\in\OO_{F^r}$.
Choose such $\lambda_i$, and take $\lambda_i\in\ZZ_{>0}$ whenever possible.
Since we have $F^r=\QQ(\sqrt{q})$ with prime $q\not\equiv 3\ (\mod\ 4)$, 
genus theory implies that $\Cl_{F^r}=\Cl^+_{F^r}$ has odd order, so $n_i$ is odd. 
Let 
\[\alpha = \prod_{i=1}^{\scriptstyle t_{K^r}}\lambda_i^{(\ord_{\lp_i}((\beta))\ \mod\ 2)}.\]

By proving the following two claims we finish the proof.

\noindent\textbf{Claim 1.} We have  $\alpha/\beta\in ({F^r}^\times)^2$.

\noindent\textbf{Claim 2.} We have $\alpha = \pi s_1\cdots s_u$ with $\pi$, $s_i$ and $u$ as in the statement.

\noindent\textit{Proof of Claim 1.} We first prove that
$({\alpha}/{\beta}) = {(\alpha)}/{(\beta)}$ is a square ideal in $F^r$. 
Let $\lp$ be any prime of $F^r$. If $\lp$ is unramified in $K^r/F^r$, then by 
(\ref{order}), we have $\ord_{\lp}((\beta)) \equiv 0\ (\mod\ 2)$. So  by the 
definition of $\alpha$, we have $\ord_{\lp}((\alpha)) = 0$.
If $\lp$ is ramified in $K^r/F^r$, then there exists $i$ such that $\lp=\lp_i$, 
so we get $$\ord_{\lp_i}((\alpha))\equiv\ord_{\lp_i}((\beta))\ord_{\lp_i}((\lambda_i)) \equiv \ord_{\lp_i}((\beta))\ (\mod\ 2)$$ 
as $n_i = \ord_{\lp_i}((\lambda_i))$ is odd. Therefore, the quotient 
$({\alpha}/{\beta})$ is a square of a fractional ideal~$\aaa$ of~$\OO_{F^r}$. Thus $\aaa^2$ 
is generated by the totally positive  $\alpha/\beta$. So the class of $\aaa$ 
is $2$-torsion in the group $\Cl^+_{F^r}$, which has an odd order,
so there is a totally 
positive element $\mu\in (F^r)^\times$ that generates~$\aaa$. 
So $\alpha/\beta = \mu^2\cdot v$ for some $v\in (\OO^{\times}_{F^r})^+$. Moreover, 
since $\Cl_{F^r}=\Cl^{+}_{F^r}$, the norm of the fundamental unit $\epsilon$ is 
negative. Therefore, a unit in $\OO_{F^r}$ is totally positive if and only if it is 
a square in $\OO_{F^r}$. Hence $v$ is a square in $\OO_{F^r}$ so we get 
$\alpha/\beta\in ({F^r}^\times)^2$. 

\noindent\textit{Proof of Claim 2.}  For any given $i$, if $l_i$ is inert in 
$F^r/\QQ$, then $n_i=1$ and $\lambda_i=l_i\in\ZZ_{>0}$ is prime. If $l_i$ is not 
inert in $F^r/\QQ$ then $l_i\in \{p, q\}$, by Proposition~\ref{prop: ram_in_Fr}. 
If $l_i=q$, then as $\lp_i$ is ramified in $K^r/F^r$, by 
Corollary~\ref{totram} we get $\sqrt{q}\in F$, contradiction. So~$l_i = p$.

Let \begin{align*}\{s_1,\ldots,s_u\} = \{l_i\ :&\ l_i\ \text{is inert in}\ F^r/\QQ\ \text{and ramified in}\ K^r/F^r\\
 &\text{and}\ \ord_{\lp_i}((\beta))\equiv 1\ (\mod\ 2)\}.\end{align*} 

Let $p\OO_{F^r}=\p\p'$. Then we have 
$\alpha=\pi^{a}\pi'^{a'}\prod_{i=1}^{u} s_{i}$, where $\pi$ 
and $\pi'$ are totally positive generators of $\p^j$ and $\p'^{j}$ for some odd
$j\in\ZZ_{>0}$. Here, we have $\prod_{i=1}^{u} s_{i}\in\ZZ_{>0}$ and 
$a,\ a'\in\{0,\ 1\}$. If $a = a'$, then $\alpha\in\ZZ$, which leads to a 
contradiction since $K^r$ is non-biquadratic. So we either have 
$\alpha=\pi\prod_{i=1}^{u} s_{i}$ or 
$\alpha=\pi'\prod_{i=1}^{u} s_{i}$. Since $\p$ and $\p'$ are lying 
above $p$ in $F^r$, they are Galois conjugates in $\OO_{F^r}$ and hence $\pi\prod_{i=1}^{u} s_{i}$ and $\pi'\prod_{i=1}^{u} s_{i}$ are 
conjugates up to a unit in $\OO_{F^r}$. As $\pi$, $\pi'$ are both totally positive,
this unit is a square in $\OO_{F^r}$. Hence $\sqrt{-\pi s_{1} \cdots s_u}$ and $\sqrt{-\pi's_{1} \cdots s_u}$ generate conjugate field
extensions, hence generate isomorphic number fields over $\QQ$.

Therefore, we have
\[K^r \cong \QQ(\sqrt{-\pi s_{1} \cdots s_u})\]
and it remains to show $u\in\{t_{K^r}-1,\ t_{K^r}-2\}$.
By (\ref{order}) we get that $u = t_{K^r}$ minus $1$
for $\pi$ (or $\pi'$), minus at most one for every prime
lying over $2$ in case $p\not=2$.
And since such primes are inert in $F^r/\QQ$,
we get that there is at most one such prime.
\end{proof}

Combining Proposition \ref{construct.1} and the bound on the discriminant in Proposition \ref{t<6}, we now have a good way of listing the fields. Next, we need a fast way of eliminating fields from our list if they have CM class number $>1$. 

The following lemma is a special case of Theorem D in Louboutin \cite{loub5}. 

\begin{lem} \label{quick}
Let $K$ be a non-biquadratic quartic CM field and let $F$ be its real quadratic subfield. Let $d_K$ and $d_F$ be the discriminants of $K$ and~$F$. Then assuming $I_{0}(\Phi^r)=I_{K^r}$,  if a rational prime $l$ splits completely in $K^r/\QQ$, then $l\geq \frac{\sqrt{d_K/{d_F}^2}}{4}$. 
\end{lem}

\begin{proof} Let $l$ be a rational prime that splits completely in $K^r/\QQ$. 
Let $\mathfrak{l}_{K^r}$ be a prime ideal in~$K^r$ above~$l$. By the 
assumption $I_{0}(\Phi^r)=I_{K^r}$, there exists $\tau\in K^{\times}$ 
such that $\N_{\Phi^r}(\mathfrak{l}_{K^r})=(\tau)$ and $\tau\overline{\tau}=l$. Here $\tau\neq\overline{\tau}$, 
since $\sqrt{l}\not\in K$. Then since $\OO_{K}\supset\OO_{F}[\tau]$ 
and $\Delta(\OO_{F}[\tau]/\OO_{F}) = (\tau-\overline{\tau})^2$, we 
have $d_{K}/d^{2}_{F} = \N_{F/\QQ} (d_{K/F}) = \N_{F/\QQ}(\Delta(\OO_{K}/\OO_{F})) \leq \N_{F/\QQ}((\tau-\overline{\tau})^2)$. 
Moreover, since $\tau\overline{\tau}=l$, we have $\phi(\tau-\overline{\tau})^2\leq (2\sqrt{l})^2$ for all embeddings $\phi: F\hookrightarrow \RR$, hence $d_{K}/d^{2}_{F} \leq \N_{F/\QQ}((\tau-\overline{\tau})^2)\leq 16l^2$.
\end{proof}

Every prime $s_i$ as in Proposition \ref{construct.1} divides $\Delta({K^r}/{F^r})$ and so $s_i^2|d_{K^r}$, hence $s_i^4|d_N$. The primes $p$ and $q$ are ramified in $F$ and $F^r$, so $p^4$ and $q^4$ divide the discriminant~$d_N$ of the normal closure $N$ of degree $8$. Hence $d_N\geq p^{4}q^{4}s^{4}_1\cdots s^{4}_{u}$ with $u\in\{t_{K^r}-1,\ t_{K^r}-2\}$. 

\begin{Algorithm} \label{algo2}

\textbf{Output:} $[D,\ A,\ B]$ representations of all non-normal quartic CM fields $K$ satisfying  $I_{0}(\Phi^r)=I_{K^r}$.
\begin{itemize}
\item[Step 1.] Find all square-free integers smaller than $3\cdot10^{10}$ having at most $8$ prime divisors and find all square-free integers smaller than $222^2$.
\item[Step 2.] Order the prime factors of each of these square-free integers
as tuples of primes $(p,q,s_1,\ldots,s_{u})$ with $s_{1}<\cdots <s_u$ in $(u+1)(u+2)$-ways.

Then keep only the tuples for which
$p$ is split in $F^r = \QQ(\sqrt{q})$,
$q$ is split in $F = \QQ(\sqrt{p})$,
and all primes $s_i$ are inert in both $F$ and~$F^r$.
\item[Step 3.] For each $(p,q,s_1,\ldots,s_{u})$, let $F^{r}=\QQ(\sqrt{q})$, write $p\OO_{F^r}=\p\p'$, and take $\alpha=\pi\cdot s_1\cdots s_u\in F^{r}$, where $\pi$ is a totally positive generator of $\p^{j}$ for the minimal $j\in\ZZ_{>0}$. Construct $K^{r}=F^{r}(\sqrt{-\alpha})$. 
\item[Step 4.] Eliminate the fields $K^r$ that have totally split primes in $K^r$ below the bound ${\sqrt{d_{K}/d^{2}_F}}/{4}$. 
\item[Step 5.]  For each $\mathfrak{q}$ with norm $Q$ below  $12\log(|d_{K^r}|)^2$,  check whether it is in $I_{0}(\Phi^r)$ as follows. List all quartic Weil $Q$--poly\-no\-mi\-als, that is, monic integer polynomials of degree $4$ such that all roots in $\CC$ have absolute value $\sqrt{Q}$. For each, take its roots in~$K$ and check whether $\N_{\Phi^r}(\mathfrak{q})$ is generated by
such a root. If not, then $\mathfrak{q}$ is not in $I_0(\Phi^r)$, so we
throw away the field.
\item[Step 6.] For each $K^r$, compute the class group of $K^r$ and for a CM type $\Phi$ of $K$ test $I_{0}(\Phi^r)/P_{K^r}=I_{K^r}/P_{K^r}$.
\item[Step 7.] Find $[D,\ A,\ B]$ representations for the reflex fields $K$ of the remaining pairs $(K^r,\Phi^r)$. 
\end{itemize}
\end{Algorithm}

\begin{proof}

Note that Step $4$ and Step $5$ of the algorithm above do not affect the validity of the algorithm by Lemma \ref{quick}. These two steps are only to speed up the computation. In Step $4$ we eliminate most of the CM fields.

Suppose that a non-normal quartic CM field $K$ satisfies $I_{0}(\Phi^r)=I_{K^r}$.
Then by Proposition~\ref{construct.1}, we have $F=\QQ(\sqrt{p})$ and $F^r=\QQ(\sqrt{q})$,
where $p$ and $q$ are prime numbers with $q\not\equiv 3\ (\mod\ 4)$ and such that
$p$ and $q$ are split in $F^r$ and $F$ respectively.
Also by Proposition~\ref{construct.1}, there exist a prime $\p$
lying above $p$ in $F^r$ that ramifies in $K^r$ and a totally positive element
$\alpha = \pi s_1\cdots s_{u}$, where~$\pi$ is a totally positive generator of $\p^{j}$
for some odd $j\in\ZZ_{>0}$ such that $K^r=F^{r}(\sqrt{-\alpha})$.
By Proposition~\ref{prop: ram_in_Fr}, the ramified primes in $K^r/F^r$ that
are distinct from $\p$ are inert in~$F$ and $F^r$. Hence $s_1, \ldots, s_{u}$
are inert in $F$ and $F^r$.
By Lemma~\ref{t<6}, we have either $h^{*}_{K^r}=2^{t_{K^r}-1}\leq2^5$
and  $d_{K^r}/d_{F^r}\leq3\cdot10^{10}$ or $d_N<222^8$.
Therefore, the CM field $K$ is listed. 
\end{proof}
We implemented the algorithm in SageMath \cite{sage, pari,streng}
and obtained the list of the fields in Theorem \ref{thm: non-normallist2}. The implementation is available online at \cite{kilicer_code}. This proves Theorems \ref{thm: non-normallist2} and \ref{thm: non-normallist}. \qed

This computation is easily parallelized
and took less than 25 core-weeks.

\begin{remk} There are no fields eliminated in Step 6, because they turned out to be already eliminated in Step 5. 
\end{remk} 

\section{The cyclic quartic case}\label{cyclic}

We apply the strategy in the previous sections in order to find all
cyclic quartic CM fields with CM class number 
one.
Murabayashi and Umegaki~\cite{umeg} already determined
those whose CM curves can be defined over~$\QQ$,
but there are additional fields with CM class number one
in Table~1b of \cite{bouyer}.

Suppose that $K/\QQ$ is a cyclic quartic CM field with $\Gal(K/\QQ)=\langle y\rangle$. Since~$K/\QQ$ is normal, we consider CM types with values in $K$. The CM type, up to equivalence, is $\Phi=\{\id,  y\}$, which is primitive. The reflex field $K^r$ is $K$ and the reflex type of $\Phi$ is the CM type $\{\id, y^{3}\}$ (Example 8.4(1) of Shimura \cite{shimura3}). In this notation complex conjugation~$\overline{\cdot}$ is $ y^2$.

Suppose $K\cong\QQ(\zeta_5)$, where $\zeta_m$ denotes a primitive $m$-th root of unity. Then the class group of $K$ is trivial, so the equality $I_{0}(\Phi^{r})=I_{K}$ holds. Hence $K=\QQ(\zeta_5)$ will occur in the list of cyclic quartic CM fields satisfying $I_{0}(\Phi^{r})=I_{K}$. 

From now on, suppose $K\not\cong\QQ(\zeta_5)$.
\begin{lem}\label{ramification1}(Murabayashi \cite{murab}, Lemma 4.2) If $I_{0}(\Phi^{r})=I_{K^r}$, then there is exactly one totally ramified prime in $K/\QQ$
	(i.e., $F=\QQ(\sqrt{p})$ for a prime $p\not \equiv 3\ (\mod\ 4)$)
	 and the other ramified primes of $K/\QQ$ are inert in $F/\QQ$. \qed
\end{lem}

\begin{example} \label{cyc2} Suppose $I_{0}(\Phi^{r})=I_{K^r}$. The relative class number $h^{*}_K$ equals $1$ if and only if $K/F$ has exactly one ramified prime.
This ramified prime is $\sqrt{p}$, where $F=\QQ(\sqrt{p})$.
\end{example}

We now determine such CM fields by using a lower bound on their relative class numbers from analytic number theory.

\begin{thm} (Louboutin \cite{loub1}, Theorem 5)\label{lower1} Let $K$ be a cyclic quartic CM field of conductor~$f_{K}$ and discriminant $d_K$. Then we have
   \begin{equation}\label{eq:cyclic}\pushQED{\qed} h^{*}_{K}\geq \frac{2}{3e \pi^2}\left(1-\frac{4\pi e^{1/2}}{d_{K}^{1/4}}\right)\frac{f_{K}}{(\log(f_{K})+0.05)^2}. \qedhere
\popQED\end{equation}	
\end{thm}

\begin{prop}\label{lemrel}
	In the situation
	of Theorem~\ref{lower1}, suppose $I_{0}(\Phi^r) = I_{K^r}$.
	Then we have $h^{*}_K\leq2^5$ and $f_K<2.1\cdot10^5$. 
\end{prop}

\begin{proof} 
Under the assumption $I_{0}(\Phi^r) = I_{K^r}$, Lemma \ref{ramification1} implies that there is exactly one totally ramified prime in $K/\QQ$ and the other ramified primes of $K/\QQ$ are inert in $F/\QQ$. 

Let $\Delta_t$ be the product of the first $t$ primes. Let $t_K$ be the number of primes in $F$ that are ramified in $K$. Since the ramified primes in $K/\QQ$ divide the conductor $f_K$, we have $f_K>\Delta_{t_K}$.  Further, by Propositions 11.9 and 11.10 in 
Chapter VII of~\cite{neukirch},
we have $d_K=f^2_K\cdot d_F$ so $d_K>\Delta_{t_K}^2$. The right hand side of (\ref{eq:cyclic}) is monotonically  increasing with $f_K>2$. Further, by Proposition \ref{prop2}, we have $h^*_K = 2^{t_K-1}$ so by dividing both sides of (\ref{eq:cyclic}) by $2^{t_K-1}$, we obtain
\begin{equation}\label{eq: cycliceqn}
1\geq \frac{2}{3e \pi^2}\left(1-\frac{4\pi e^{1/2}}{\Delta_{t_K}^{1/2}}\right)\frac{\Delta_{t_K}}{2^{t_K}(\log(\Delta_{t_K})+0.05)^2}.
\end{equation}
The right hand side of  (\ref{eq: cycliceqn}) is monotonically increasing with $t_K\geq 2$, and if $t_K=7$, then the right hand side is greater than $1$. Hence $t_K\leq 6$. So we get $h^{*}_{K}\leq 2^5$, and therefore, we get $f_K< 2.1\cdot10^5$. 
\end{proof}

\begin{thm}\label{thm: cycliclist2} There exist exactly $20$ isomorphism classes of CM class number one cyclic quartic CM fields. The fields are given by $K\cong\QQ[X]/(X^4 + AX^2 + B)\supset \QQ(\sqrt{D})$ where $[D,A,B]$ ranges over
\[[5,5,5],\ [8,4,2],\ [13,13,13],\ [29, 29, 29],\] \[[37,37,333],\ [53, 53, 53],\ [61, 61, 549] \] 
with class number $1$,
\[[5, 65, 845],\ [5, 85, 1445],\ [5, 10, 20],\ [8, 12, 18],\]\[ [8, 20, 50],\ [13, 65, 325],\ [13, 26, 52],\ [17, 119, 3332] \]
with class number $2$, and 
\[[5, 30, 180],\ [5, 35, 245],\ [5, 15, 45],\ [5, 105, 2205],\ [17, 255, 15300]\]
with class number $4$.
\end{thm}

We begin the proof with the following proposition.
\begin{prop}\label{construct.2}
If a cyclic quartic CM field $K$ satisfies $I_{0}(\Phi^r)=I_{K}$, then there exist prime numbers $p$, $s_1, \ldots, s_u$ $\in\ZZ$ such that $F=\QQ(\sqrt{p})$ with $p\not\equiv 3\ (\mod\ 4)$ and $s_i$ is inert in $F$ for all $i$, and we have $K^r\cong\QQ(\sqrt{-\epsilon s_1\cdots s_{u}\sqrt{p}})$ with $u\in\{t_K-1,\ t_K-2\}$ for every $\epsilon\in\OO^{\times}_{F}$ with $\epsilon\sqrt{p}\gg0$. 
\end{prop}

\begin{proof} By Proposition \ref{ramification1}, we have $F=\QQ(\sqrt{p})$, where $p$ is a prime with $p\not\equiv3\ (\mod\ 4)$. If there are $t_K$ ramified primes in $K/F$, the ones that are distinct from the one above~$p$ are inert in $F/\QQ$, by Proposition \ref{ramification1}; denote them by
$s_1, \ldots, s_{t_K-1}$.
	
There exists a totally positive element $\beta$ in ${F}^\times $ (without loss of generality, we can take $\beta$ in $\OO_{F}$) such that $K=F(\sqrt{-\beta})$, where $\beta$ is uniquely defined up to $({F}^\times )^2$. 
As in the proof of Proposition~\ref{construct.1} in the previous section, 
we define a totally positive element $\alpha\in F^\times$
with respect to the ramified 
primes in $K/F$ and show that $\alpha$ and~$\beta$ differ by a factor in $(F^\times )^2$. 

Let $\epsilon\in\OO^{\times}_{F}$ such that $\epsilon\sqrt{p}\gg0$. Such an element exists 
since $p\not\equiv 3\ (\mod\ 4)$. As~$\beta$ is unique up to squares and we can take $\lp$-
minimal $\beta'\in\beta(F^\times )^2$ for each prime $\lp$~of $\OO_F$,
we get
\begin{equation}\label{eq:ordercyclic} \ord_{\lp}((\beta))\equiv \begin{cases} 
1\ (\mod\ 2)&\quad \text{if $\lp$ is ramified in $K^r/F^r $ and $\lp\nmid 2$,}\\
0\ (\mod\ 2) &\quad \text{if $\lp$ is not ramified in $K^r/F^r $,}\\
0\ \text{or}\ 1\ (\mod\ 2)&\quad \text{if $\lp$ is ramified in $K^r/F^r $ and $\lp|2$.}
\end{cases} 
\end{equation}

If there is an $i< t_K$ with $s_i=2$, then without loss of
generality, take $s_{t_K-1} = 2$.
If $p\neq 2$ and the prime $(2)$ in $\OO_F$ is ramified in $K/F$ with 
$\ord_{(2)}((\beta))\equiv 0\ (\mod\ 2)$, then take 
$\alpha:= \epsilon s_1\cdots s_{u}\sqrt{p}$ 
with $u=t_K-2$. If $p = 2$ and $\ord_{(\sqrt{2})}((\beta))\equiv 0\ (\mod\ 2)$, 
then take $\alpha:= s_1\cdots s_{u}$ with $u=t_K-1$. For all other cases in (\ref{eq:ordercyclic}), 
take $\alpha:= \epsilon s_1\cdots s_{u}\sqrt{p}$ with $u=t_K-1$. 

By the definition of $\alpha$, for all ideals $\lp \subset \OO_F$ we have 
$\ord_{\lp}((\alpha/\beta))\equiv 0\ (\mod\ 2)$. So $(\alpha/\beta)=\aaa^2$ for a 
fractional $\OO_F$-ideal $\aaa$. The ideal $\aaa$ is a $2$-torsion element in~$\Cl_F$. 
Moreover, since $F=\QQ(\sqrt{p})$ with $p\not\equiv 3\ (\mod\ 4)$, genus theory implies that 
$\Cl_F=\Cl^{+}_F$ has odd order. Therefore, there is a totally positive element $\mu$ that 
generates $\aaa$. So $\alpha/\beta=\mu^2\cdot v$ for some $v\in \OO^{+}_F$. Furthermore, 
since $\Cl_F=\Cl^{+}_F$, the fundamental unit has a negative norm, and so 
$\OO^{+}_F=(\OO_F)^2$. Hence $v$ is a square in~$\OO_F$, and therefore we get 
$\alpha/\beta\in (F^{\times})^2$. 

In the case $p = 2$ and $\ord_{(\sqrt{2})}((\beta))\equiv 0\ (\mod\ 2)$, we get the 
biquadratic field $K=F(\sqrt{- s_1\cdots s_u})$ over $\QQ$, contradiction. Therefore, we have
\[K = \QQ(\sqrt{-\epsilon s_1\cdots s_{u}\sqrt{p}})\  
\text{with}\   u\in\{t_{K-1},\ t_{K-2}\}.\qedhere\]
\end{proof}

\begin{Algorithm}
\textbf{Output:} $[D,\ A,\ B]$ representations of all  cyclic quartic CM \mbox{fields $K$} satisfying  $I_{0}(\Phi^r)=I_{K}$.
\begin{itemize}
\item[Step 1.] Find all square-free integers less than $2.1\cdot10^5$
having at most $6$ prime divisors.
\item[Step 2.] Order the prime factors of each of these square-free integers as tuples of primes $(p,s_1,\ldots, s_{u})$ with $s_1<\cdots <s_u$ in $(u+1)$-ways, then take only the tuples satisfying $p\not\equiv3\ (\mod\ 4)$ and that $s_i$ is inert in $\QQ(\sqrt{p})$ for all $i$.
\item[Step 3.] For each $(p,s_1,\ldots, s_{u})$, let $F=\QQ(\sqrt{p})$ and take a totally positive element  $\alpha = \epsilon s_1 \cdots s_u\sqrt{p}$, where $\epsilon$ is a fundamental unit in $F$ such that  $\epsilon\sqrt{p}\gg0$. Construct $K=F(\sqrt{-\alpha})$. 
\item[Step 4.] Eliminate the fields $K$ that have totally split primes in $K$ below the bound ${\sqrt{d_{K}/d^{2}_F}}/{4}$. (In this step we eliminate most of the CM fields.)
\item[Step 5.] For each $\mathfrak{q}$ with norm $Q$ below the bound $12\log(|d_{K^r}|)^2$,  check whether it is in $I_{0}(\Phi^r)$ as follows. List all quartic Weil $Q$-polynomials, that is, monic integer polynomials of degree $4$ such that all roots in $\CC$ have absolute value~$\sqrt{Q}$. For each, take its roots in $K$ and check whether $\N_{\Phi^r}(\mathfrak{q})$ is generated by
such a root. If not, then $\mathfrak{q}$ is not in $I_0(\Phi^r)$, so we
throw away the field.
\item[Step 6.] For each $K$ compute the class group of the fields $K$ and test $I_{0}(\Phi^r)/P_{K}=I_{K}/P_{K}$.
\item[Step 7.] Find $[D,\ A,\ B]$ representations for the fields~$K$. 
\end{itemize}
\end{Algorithm}

\begin{proof} The idea of the proof of this algorithm is exactly as the proof of
Algorithm~\ref{algo2}. In this algorithm, Step 1 follows from Proposition \ref{lemrel}, Step 2 
and 3 follow from Proposition~\ref{construct.2}, and Step 4 follows from Lemma \ref{quick}.
\end{proof}

We implemented the algorithm in SAGE~\cite{sage, pari,streng} and obtained the list of the fields in Theorem~\ref{thm: cycliclist2}. The implementation is available online at \cite{kilicer_code}. This proves Theorems~\ref{thm: cycliclist} and \ref{thm: cycliclist2}. \qed

This computation is easily parallelized and took less than 4 core-weeks.

\bibliographystyle{plain}
\bibliography{mybib}

\Addresses
\end{document}